\documentclass[reqno, 10pt]{amsart}

\usepackage{amsthm,amsmath,amsfonts,amssymb,epsfig,graphicx,latexsym,float,caption}
\usepackage{xcolor}

\newcommand {\mat}      [1] {\left[\begin{array}{#1}}
\newcommand {\rix}          {\end{array}\right]}
\newtheorem{definition}{Definition}
\newtheorem{theorem}{Theorem}
\newtheorem{lemma}{Lemma}
\newtheorem{remark}{Remark}

\begin{document}

\title[MOR for pHDAE]
{\fontsize{12}{15}\selectfont Model reduction techniques for linear constant coefficient port-Hamiltonian differential-algebraic systems}

\author[S.~Hauschild]{Sarah-Alexa Hauschild$^{1,\star}$}
\author[N.~Marheineke]{Nicole Marheineke$^{1}$}
\author[V.~Mehrmann]{Volker Mehrmann$^{2}$}

\date{\today\\
$^\star$ \textit{Corresponding author}, email: hauschild@uni-trier.de, phone: +49\,651\,201\,4788\\
$^1$ Universit\"at Trier, Lehrstuhl Modellierung und Numerik, Universit\"atsring.\ 15, D-54296 Trier, Germany\\
$^2$ TU Berlin, Institut f\"ur Mathematik, Stra{\ss}e des 17.\ Juni 136, D-10623 Berlin, Germany
}

\begin{abstract}\normalsize
Port-based network modeling of multi-physics problems leads naturally to a formulation as port-Hamiltonian differential-algebraic system. In this way, the physical properties are directly encoded in the structure of the model. Since the state space dimension of such systems may be very large, in particular when the model is a space-discretized partial differential-algebraic system, in optimization and control there is a need for model reduction methods that preserve the port-Hamiltonian structure while keeping the (explicit and implicit) algebraic constraints unchanged. To combine model reduction for differential-algebraic equations with port-Hamiltonian structure preservation, we adapt two classes of techniques (reduction of the Dirac structure and moment matching) to handle port-Hamiltonian differential-algebraic equations. The performance of the methods is investigated for benchmark examples originating from  semi-discretized flow problems and mechanical multibody systems.
\end{abstract}

\maketitle

{\sc AMS-Classification. 34H05, 41A22,  65L80, 93A15, 65F22} \\

\section{Introduction}\label{sec_0_intro}
\noindent
Port-Hamiltonian differential-algebraic systems (pHDAEs) arise from port-based network modeling of multi-physics problems. For this, a physical system is decomposed into smaller subsystems that are interconnected through energy exchange. The subsystems may belong to various different physical domains, such as electrical, mechanical, or hydraulic ones. The energy-based formulation is advantageous as it brings different scales on a single level, the port-Hamiltonian character is inherited by the coupling and the physical properties (e.g., stability, passivity, energy and momentum conservation) are encoded directly in the structure of the pH model equations \cite{BeaMXZ18, Sch13}. Algebraic constraints naturally come from the interconnections in form of network conditions, such as Kirchhoff's laws, or from constraints that are directly modeled, like position or velocity constraints in mechanical systems, or from mass balances in chemical engineering problems, see e.g., \cite{BreCP96, KunM06, Ria08}. The state space dimension of pHDAEs can be very large, e.g., for constraint finite element models in structural mechanics \cite{GraMQSW16}, semi-discretized problems arising in fluid dynamics \cite{EggKLMM18, EmmM13, HeiSS08, Sty06} or multibody problems \cite{MehS05}. In this case, for optimization and control, model order reduction techniques are needed that preserve the port-Hamiltonian structure and keep the explicit and hidden algebraic constraints unchanged. To present such methods and to study their properties is the main topic of this paper, which brings together model reduction for differential-algebraic equations with structure preservation.

The properties of pHDAEs have recently been studied in \cite{BeaMXZ18, Sch13}.  For systems of port-Hamiltonian ordinary differential equations (pHODEs), structure-preserving reduction methods have been developed based on ideas of tangential interpolation \cite{GugPBS12, GugPBS09}, moment matching \cite{PolS10, PolS11, WolLEK10} as well as effort and flow constraint reduction methods \cite{PolS12}. Structure-preserving model reduction for nonlinear systems has been studied in \cite{ChaBG16} and for linear damped wave equations in \cite{EggKLMM18}, where particular Galerkin projections have been constructed for the pHDAEs arising in gas transport networks. Surveys of model reduction techniques for general DAEs are given in \cite{BenMS05, BenS17}. A crucial step for model reduction is that the dynamic and algebraic equations are exactly identified and only the dynamic equations are reduced, otherwise the system may loose important properties, such as stability or passivity.

In this paper we generalize structure-preserving techniques that have been developed for pHODEs to pHDAEs. We focus on the effort and flow constraint model reduction methods as well as on moment matching. To do this, we follow the regularization concept of \cite{BeaMXZ18} to identify and decouple the algebraic constraints and the dynamical equations in a structure-preserving manner to develop the corresponding reduction methods. To illustrate the performance of the reduction methods, we apply them  to benchmark problems originating from  semi-discretized flow calculations and  multibody systems.  We discuss the advantages but also the limitations of these methods for pHDAEs.

This paper is structured as follows. The model framework of port-Hamiltonian differential-algebraic systems and a structure-preserving regularization concept is presented in Section~\ref{sec:2}. Following this, structure-preserving reduction techniques for the differential-algebraic pH systems are  generalized from their ordinary differential equation counterparts in Section~\ref{sec:3}. The performance of the methods is numerically investigated on the basis of various benchmark examples in Section~\ref{sec:4}. The paper closes with a summary in Section~\ref{sec:5}.

\section{Model framework of port-Hamiltonian differential-algebraic systems}\label{sec:2}
\noindent
In this section we review the structural properties  and simplified representations of pHDAEs according to \cite{BeaMXZ18} and \cite{SchM18}.  In particular, we study the decoupling of the dynamic and algebraic equations and variables, that will be used in the next section to derive structure-preserving model reduction techniques.

\subsection{Port-Hamiltonian systems}
\noindent Port-Hamiltonian systems  can be derived in two different ways, via a formulation as descriptor systems with special structured coefficient matrices or via an energy-based formulation on top of a Dirac structure. Since each formulation is a basis for a model reduction technique, we briefly discuss their relation.
\begin{definition}[pHDAE]\label{pHDAEDEF} A linear constant coefficient DAE system of the form
\begin{equation}\label{BEAMXZ18}
\begin{aligned}
E\dot{x}&=(J-R)\,Qx+(B-P)\,u,\\
y&=(B+P)^T\,Qx+(S+N)\,u,
\end{aligned}
\end{equation}
with $E$, $Q$, $J$, $R\in\mathbb{R}^{n\times n}$, $B$, $P\in\mathbb{R}^{n\times m}$, $S=S^T$, $N=-N^T\in\mathbb{R}^{m\times m}$, on a compact interval $\mathbb I\subset \mathbb R$ is called a \emph{port-Hamiltonian differential-algebraic system (pHDAE)} if the following properties are satisfied.
\begin{enumerate}
\item The differential-algebraic operator
\begin{align*}
Q^TE\frac{d}{dt}-Q^TJQ : \mathcal{X}\subset \mathcal{C}^1(\mathbb{I},\mathbb{R}^n)\rightarrow \mathcal{C}^0(\mathbb{I},\mathbb{R}^n)
\end{align*}
is \emph{skew-adjoint}, i.e., we have that
$Q^TJ^TQ=-Q^TJQ$ and $Q^TE=E^TQ$,
\item the product $Q^TE=E^TQ$ is positive semidefinite, i.e., $Q^TE=E^TQ\geq 0$, and
\item the \emph{passivity} matrix
\begin{align*}
W=\begin{bmatrix}Q^TRQ & Q^TP\\ P^TQ & S\end{bmatrix}\in\mathbb{R}^{(n+m)\times(n+m)}
\end{align*}
is symmetric positive semi-definite, i.e., $W=W^T\geq 0$.
\end{enumerate}
The quadratic \emph{Hamiltonian function} ${\mathcal H}:\mathbb{R}^n\rightarrow\mathbb{R}$ of the system is given by
\begin{align}\label{HamDae}
{\mathcal H}(x)=\frac{1}{2}x^TQ^TEx.
\end{align}
\end{definition}
\begin{theorem}\label{th:pHDAE}
Consider a pHDAE of the form \eqref{BEAMXZ18}. If for given input function $u$ the system has a (classical) solution $x\in \mathcal{C}^1(\mathbb{I},\mathbb{R}^n)$ in $\mathbb{I}$, then
\begin{align*}
\frac{d}{dt}{\mathcal H}(x)=u^Ty-\begin{bmatrix}x\\u\end{bmatrix}^TW\begin{bmatrix}x\\u\end{bmatrix}
\end{align*}
Furthermore, if $W=0$, then $\frac{d}{dt}{\mathcal H}(x)=u^Ty$.
\end{theorem}
Theorem~\ref{th:pHDAE} implies some important properties of a pHDAE. First of all, its Hamiltonian is an \emph{energy storage function}, and the system is passive. A pHDAE satisfies a dissipation inequality. Furthermore, it is implicitly \emph{Lyapunov stable} as $\mathcal H$ defines a \emph{Lyapunov function}. The physical properties are encoded in the algebraic structure of the coefficient matrices and the geometric structures associated with the flow of the system. In this sense, $E^TQ$ is the \emph{energy matrix}, $Q^TRQ$ is the \emph{dissipation matrix}, $Q^TJQ$ the \emph{structure matrix} describing the energy flux among the energy storage elements, $B  \pm P$ are the \emph{port matrices} for energy in-  and output, and  $S$, $N$ are the matrices associated to a \emph{direct feed-through} from input $u$ to output $y$. In the case that $E=I$ is the identity matrix, the pHDAE reduces to a standard pHODE as studied in \cite{SchJ14}.

In the alternative \emph{energy-based formulation} a port-Hamiltonian system is characterized by the fact that the \emph{flow and effort variables} of its energy-storing port, its energy-dissipating port and its external port are linked together in a power-conserving manner by a \emph{Dirac structure}. Given a finite-dimensional linear space $\mathcal{F}$ with its dual space $\mathcal{E}=\mathcal{F}^*$, a Dirac structure is a subset $\mathcal{D}\subset\mathcal{F}\times\mathcal{E}$ satisfying $e^Tf=0$ for all $(f,e)\in\mathcal{D}$ and $\dim\mathcal{D}=\dim\mathcal{F}$. For a pHDAE in the form  \eqref{BEAMXZ18}, the flow and effort variables are defined on $\mathcal{F}=\mathcal{F}_x\times \mathcal{F}_R \times \mathcal{F}_P$ and $\mathcal{E}=\mathcal{E}_x\times \mathcal{E}_R \times \mathcal{E}_P\subset \mathbb{R}^n \times \mathbb{R}^{n+m} \times \mathbb{R}^m$, respectively. They are given by
\begin{align*}
&f_x=-E\dot x, \quad e_x=Qx, \quad f_R=- \begin{bmatrix}R & P\\P^T&S \end{bmatrix} e_R, \quad e_R=\begin{bmatrix}Qx\\u\end{bmatrix}, \quad f_P=y, \quad e_P=u,\\
&((f_x,f_R,f_P),(e_x,e_R,e_P))\in \mathcal{D}\subset \mathcal{F}\times \mathcal{E} \text{ for all }t \in \mathbb{I}.
\end{align*}
The variables $(f_x,e_x)\in \mathcal{F}_x\times \mathcal{E}_x$ of the energy-storing port are related to the evolution of the state and the Hamiltonian $\mathcal H$. If $E=I$, then the constitutive relations  read  as $(f_x,e_x)=(-\dot x, \nabla_x {\mathcal H}(x))$. The port variables of the energy-dissipating elements satisfy a resistive relation, $e_R^Tf_R\leq 0$ for all $(f_R,e_R)\in \mathcal{R} \subset\mathcal{F}_R\times\mathcal{E}_R$, which is encoded in the stated positive semidefinite matrix $R$. The external port variables $(f_P,e_P)\in \mathcal{F}_P\times \mathcal{E}_P$ correspond to the out- and inputs of the system. The energy-conservation property follows directly from Theorem~\ref{th:pHDAE}. Based on the notion of the Dirac structure, the pH system possesses a DAE representation, see \cite{SchJ14, SchM18}, i.e.,
\begin{align}\label{DAErep}
-F_x\,f_x=E_x \,e_x+F_R\,f_R+E_R\,e_R+F_P\,f_P+E_P\,e_P && \text{ for all }t \in \mathbb{I}
\end{align}
with matrices $F_x,E_x\in\mathbb{R}^{q\times n},$ $F_R,E_R\in\mathbb{R}^{q\times (n+m)}$ and $F_P,E_P\in\mathbb{R}^{q\times m}$ where $q=n+(n+m)+m$ and $\sum_{i=x,R,P} E_i F_i^T+ F_i E_i^T =0$.

\subsection{Structure-preserving regularization}\label{sec:decoupling}
\noindent A pHDAE system typically contains explicit as well as implicit (hidden) constraints. Since in model reduction all constraints need to be kept unchanged in order not to destroy crucial properties, we need to identify all constraints. If the differentiation-index is larger than one, then an index reduction, e.g., via derivative arrays or minimal extension, should be performed, see \cite{KunM06} for general DAEs. For pHDAEs this index reduction has to be performed in a structure-preserving way, see \cite{BeaMXZ18}.  It has been shown in  \cite{MehMW18} for the linear constant coefficient case and in \cite{Sch17_ppt} for the linear time-varying case that the differentiation-index will be at most two, i.e., in simple terms at most the second derivative of the input function $u$ is required to transform the system into a pHODE.
In contrast to the numerical solution of pHDAEs/pHODEs via time-integration methods, which is still partially an open problem, see, e.g. \cite{KotL18}, in the context of model reduction also a structure-preserving decoupling of the dynamic and algebraic variables should be performed. This may be a very critical step for linear time-varying or nonlinear systems, since it may require time-varying changes of variables, with all its difficulties, in particular of having to provide derivatives of the transformation functions \cite{KunM06}. But even in the case of constant coefficients, changes of variables with ill-conditioned transformation matrices may have to be handled.

In the following we discuss the structure-preserving regularization for linear constant coefficient pHDAEs of differentiation-index one or two. We refer to \cite{BeaMXZ18} for a detailed analysis of the regularization concept for linear time-varying systems. The concept is particularly based on the fact that the port-Hamiltonian structure and the associated Hamiltonian are preserved under basis change and scaling with invertible matrices (cf.\ Lemma~\ref{TransfoTheo}).
\begin{lemma}\label{TransfoTheo}
Consider a pHDAE of the form \eqref{BEAMXZ18} with Hamiltonian $\mathcal H$ \eqref{HamDae}. Let $U$, $V\in\mathbb{R}^{n\times n}$ be invertible. Then the transformed system
\begin{align*}
\tilde{E}\dot{\tilde{x}}&=(\tilde{J}-\tilde{R})\tilde{Q}\tilde{x}+(\tilde{B}-\tilde{P})u,\\
y&=(\tilde{B}+\tilde{P})^T\tilde{Q}\tilde{x}+(S+N)u
\end{align*}
with $\tilde{E}=U^TEV$, $\tilde{J}=U^TJU$, $\tilde{R}=U^TRU$, $\tilde{B}=U^TB$, $\tilde{P}=U^TP$, as well as $\tilde{Q}=U^{-1}QV$ and $\tilde{x}=V^{-1}x$ is still a pHDAE with the same Hamiltonian $\tilde{\mathcal H}(\tilde{x})=\frac{1}{2}\tilde{x}^T\tilde{Q}^T\tilde{E}\tilde{x}={\mathcal H}(x)$.
\end{lemma}

\subsubsection*{Decoupling of pHDAE of index at most one}
For the decoupling of a pHDAE \eqref{BEAMXZ18} of index at most one, two orthogonal matrices $\tilde{U}$, $V\in \mathbb{R}^{n\times n}$ are determined (e.g., via a singular decomposition) such that
\begin{equation}\label{SVDE}
\tilde{U}^TEV=\begin{bmatrix}E_{11} & 0\\0&0\end{bmatrix}
\end{equation}
with $E_{11}$ invertible. We set $L=J-R$ and apply the transformation induced by $\tilde{U}$ and $V$ to \eqref{BEAMXZ18} (cf.\ Lemma~\ref{TransfoTheo}). Partitioning as in \eqref{SVDE} yields block-structured system matrices whose blocks we denote by $\tilde{\quad}$ in case that they change in the decoupling procedure. As $Q^TE$ is real symmetric, we have $Q^T_{11}E_{11}=E^T_{11}Q_{11}$ and $Q_{12}=0$. Furthermore, as the system is of differentiation-index at most one, the block matrix $L_{22}Q_{22}$ is either not present -- in case of an implicitly defined pHODE -- or it is invertible, i.e., $L_{22}$ and $Q_{22}$ both are invertible.
Setting $U=\tilde{U}T$ with
\begin{align*}
T=\begin{bmatrix} I&0\\T_{21}&I\end{bmatrix}, \quad T_{21}=-(L_{22})^{-T}(\tilde{L}_{12})^T
\end{align*}
and transforming \eqref{BEAMXZ18} with $U$ and $V$ as in Lemma~\ref{TransfoTheo} yields the  block-structured pHDAE
\begin{equation}\label{Index1Transf}
\begin{aligned}
\begin{bmatrix}E_{11}&0\\0&0\end{bmatrix}\begin{bmatrix}\dot{x}_1\\ \dot{x}_2\end{bmatrix}&=\begin{bmatrix}{L}_{11}
&0\\{L}_{21}&L_{22}\end{bmatrix}\begin{bmatrix}Q_{11}&0\\{Q}_{21}&Q_{22}\end{bmatrix}\begin{bmatrix}x_1\\x_2
\end{bmatrix}+\left(\begin{bmatrix}B_1\\B_2\end{bmatrix}-
\begin{bmatrix}P_1\\P_2\end{bmatrix}\right)u\\
y&=\left(\begin{bmatrix}B_1\\B_2\end{bmatrix}+\begin{bmatrix}P_1\\P_2\end{bmatrix}\right)^T
\begin{bmatrix}Q_{11}&0\\Q_{21}&Q_{22}\end{bmatrix}
\begin{bmatrix}x_1\\x_2\end{bmatrix}+(S+N)u,
\end{aligned}
\end{equation}
where $L_{ij}=J_{ij}-R_{ij}$.

\begin{theorem}[Decoupled pHDAE]\label{pHDAE1st}
Suppose that the pHDAE \eqref{BEAMXZ18} is of differentiation-index at most one. Let the system be transformed to the form \eqref{Index1Transf} via $U$ and $V$ and define $V^{-1}x=[x_1^T \,\, x_2^T]^T$. Then for any input $u$ and initial condition $x_1(t_0) = x_{1,0}$, the output $y$ and the state $x_1$ of \eqref{Index1Transf} are given by the implicit pHODE
\begin{equation*}
\begin{aligned}
E_{11}\,\dot{x}_1&=({J}_{11}-{R}_{11})Q_{11}\,x_1+(\hat{B}_1-\hat{P}_1)\,u,\\
y&=(\hat{B}_1+\hat{P}_1)^TQ_{11}\,x_1+(\hat{S}+\hat{N})\,u
\end{aligned}
\end{equation*}
with Hamiltonian $\hat{\mathcal H}(x_1)=\frac{1}{2}x_1^TQ^T_{11}E_{11}x_1={\mathcal H}(x)$, and coefficients
\begin{align*}
\hat{B}_1&=B_1-\frac{1}{2} L^T_{21} L_{22}^{-T}(B_2+P_2),\\
\hat{P}_1&=P_1-\frac{1}{2} L^T_{21} L_{22}^{-T}(B_2+P_2),\\
\hat{S}&=S-\frac{1}{2}[(B_2+P_2)^TL_{22}^{-1}(B_2-P_2)+(B_2-P_2)^TL_{22}^{-T}(B_2+P_2)],\\
\hat{N}&=N-\frac{1}{2}[(B_2+P_2)^TL_{22}^{-1}(B_2-P_2)-(B_2-P_2)^TL_{22}^{-T}(B_2+P_2)].
\end{align*}
The state $x_2$ is uniquely determined by the explicit algebraic constraint
\begin{align*}
L_{22}Q_{22}\,x_2=-({L}_{21}Q_{11}+L_{22}{Q}_{21})\,x_1-(B_2-P_2)\,u,
\end{align*}
which implies a consistency constraint for the respective initial condition.
\end{theorem}
Typically the original pencil $(s E-LQ)$, $s\in \mathbb{C}$, is regular, i.e., its determinant is not identically zero, which means that a unique  solution exist for every sufficiently smooth input function $u$ and every consistent initial condition. If this is not the case, then a complex regularization procedure can be performed, which consists of transformations, feedbacks and renaming of variables, see \cite{CamKM12}. Since this procedure is not yet available for pHDAEs, in the following we assume that $(s E-LQ)$ is regular. Then the pencil $(s E-Q)$ is regular as well as shown in \cite{MehMW18}. In this case \eqref{Index1Transf} can be decoupled even further by identifying the zero eigenvalues of the system.  For this, a change of basis is applied to the dynamic state $x_1$. From $E^TQ=Q^TE\geq 0$ it follows that the block matrix $Q_{11}E_{11}^{-1}$ is symmetric positive semidefinite and allows for an ordered Schur decomposition that can be obtained from the generalized singular value decomposition \cite[]{GolV96}. Even though we will not carry out this transformation explicitly, it follows that the system can be transformed as
\begin{align*}
{Q}_{11}E_{11}^{-1}=\bar{U}\begin{bmatrix}\Sigma_Q&0\\0&0\end{bmatrix}\bar{U}^T, \qquad \Sigma_Q=\Sigma_Q^T>0, \qquad \bar U^T=\bar U^{-1}.
\end{align*}
Introducing $\bar{U}^TE_{11}x_1=[x_{1,a}^T\,\, x_{1,b}^T]^T$ yields then a pHDAE of the form
\begin{align}\label{DAE_Q_Zeros}
\begin{bmatrix}I&0&0\\0&I&0\\0&0&0\end{bmatrix}\begin{bmatrix}
\dot{x}_{1,a}\\\dot{x}_{1,b}\\\dot{x}_2
\end{bmatrix}&=\begin{bmatrix}L_{11}^{aa} & L_{11}^{ab} & 0\\ L_{11}^{ba} & L_{11}^{bb} & 0\\ {L}_{21}^a\, & {L}_{21}^b\, & L_{22}
\end{bmatrix}\begin{bmatrix}\Sigma_Q &0&0\\0&0&0\\{Q}_{21}^a&{Q}_{21}^b&Q_{22}\end{bmatrix}\begin{bmatrix}
x_{1,a}\\x_{1,b}\\x_2\end{bmatrix}+\left(\begin{bmatrix}{B}_{1}^a\\ {B}_{1}^b\\ B_2
\end{bmatrix}-\begin{bmatrix}{P}_{1}^a\\ {P}_{1}^b\\ P_2\end{bmatrix}\right)u,\\
y&=\left(\begin{bmatrix}{B}_{1}^a\\ {B}_{1}^b\\ B_2
\end{bmatrix}-\begin{bmatrix}{P}_{1}^a\\ {P}_{1}^b\\ P_2\end{bmatrix}\right)^T\begin{bmatrix}\Sigma_Q &0&0\\0&0&0\\{Q}_{21}^a&{Q}_{21}^b&Q_{22}\end{bmatrix}\begin{bmatrix}
x_{1,a}\\x_{1,b}\\x_2\end{bmatrix}+(\hat{S}+\hat{N})u. \nonumber
\end{align}
\begin{remark}\label{rem1}{\rm
In the decoupled form \eqref{DAE_Q_Zeros}, some further transformations can be applied to achieve $Q_{21}^a=0$. However, since the inverse of $Q_{22}$ is involved in this transformation, we stay with the form \eqref{Index1Transf} (\eqref{DAE_Q_Zeros}, respectively).}
\end{remark}

\subsubsection*{Decoupling for pHDAE of index two}
To decouple the algebraic and differential variables for a pHDAE of differentiation-index two
we make use of the index-reduction procedure developed in \cite{BeaMXZ18} for the linear time-varying case. Assume that the state equation with $u=0$ forms a DAE of index two.  It has been shown in \cite{ByeGM97} that the extra constraint equations (hidden constraints) that arise from derivatives are uncontrollable, because otherwise the index reduction could have been done via feedback. This means that these hidden constraint equations are of the form
$\hat A x=0$. We just add these constraint equations to our original
pHDAE \eqref{BEAMXZ18} and obtain an overdetermined DAE system, see also \cite{KunM06}.
Then we perform a singular value decomposition of $E$ by means of orthogonal matrices $\tilde U_1$, $\tilde V_1\in \mathbb{R}^{n\times n}$ as in \eqref{SVDE} and partition the matrix associated to the extra constraints accordingly ${\hat A} \tilde V_1= [A_{1} \,\, \tilde A_{2}]$. Note that we denote blocks by $\tilde{\quad}$ in case that they change in the reduction procedure.
The equations  $\hat A x=0$ include all the constraints that are needed for index reduction. Since  $\tilde E_{11}$ is invertible, these extra equations must arise from the full row-rank part of $\tilde A_{2}$.  In the following we assume w.l.o.g.\ that $\tilde A_{2}$ has full row-rank. This can be always achieved by transforming and omitting hidden constraint equations that do not contribute to the index reduction, see \cite{BeaMXZ18}. Then there exists an orthogonal matrix $\tilde{V}_2$ such that $\tilde A_{2} \tilde{V}_2=[0 \,\, A_3]$ with $A_{3}$ invertible. Introducing
$[x_1^T \,\,x_2^T \,\, x_3^T]^T=V^{-1}x$ with
\begin{align*}
V=\tilde V_1 \mat{cc} I & 0 \\ 0 & \tilde V_2 \rix \mat{ccc} I & 0 & 0 \\ 0 & I & 0 \\ -A_{3}^{-1} A_{1}& 0 & I \rix,
\end{align*}
the hidden constraint equations become $A_{3}x_3=0$, consequently it follows that $x_3=0$. In addition we use an orthogonal matrix $\tilde{U}_2$ such that
\begin{align*}
\tilde{U}_2^T (\tilde{U}_1^T Q V) =\begin{bmatrix} Q_{11} & Q_{12} &  Q_{13} \\Q_{21} & Q_{22} & Q_{23}   \\  0 & 0& Q_{33}  \end{bmatrix},
\end{align*}
transforming \eqref{BEAMXZ18} with $U=\tilde{U}_1\tilde{U}_2$ and $V$ as in Lemma~\ref{TransfoTheo} yields then the following block-structured pHDAE
\begin{align*}
\nonumber
\begin{bmatrix}
E_{11} & 0 &0\\
E_{21} & 0 & 0  \\ E_{31} & 0 & 0\end{bmatrix}
\begin{bmatrix}
\dot x_1 \\ \dot x_2 \\ \dot x_3
\end{bmatrix}&=
\begin{bmatrix}
L_{11}&L_{12} &  L_{13} \\ L_{21}&L_{22}  & L_{23}\\ L_{31}&L_{32} &  L_{33}
\end{bmatrix}
\begin{bmatrix}
Q_{11} & Q_{12} &  Q_{13} \\Q_{21} & Q_{22} & Q_{23}   \\  0 & 0& Q_{33}
\end{bmatrix}\begin{bmatrix}
x_1 \\ x_2 \\ x_3\end{bmatrix}+
\left(\begin{bmatrix}B_1\\B_2\\B_3\end{bmatrix}-\begin{bmatrix}P_1\\P_2 \\ P_3\end{bmatrix}\right)u,\\
y&=
\left(\begin{bmatrix}B_1\\B_2\\B_3\end{bmatrix}+\begin{bmatrix}P_1\\P_2 \\ P_3\end{bmatrix}\right)^T
\begin{bmatrix}
Q_{11} & Q_{12} &  Q_{13} \\ Q_{21} & Q_{22} & Q_{23}   \\  0 &  0& Q_{33}\end{bmatrix}\begin{bmatrix}
x_1 \\ x_2 \\ x_3\end{bmatrix} +(S+N)u
\end{align*}
together with the constraint $x_3=0$, where $L_{ij}=J_{ij}-R_{ij}$. Since the constraint does not change the solution, the subsystem given by the first two block rows is a pHDAE of differentiation-index at most one, i.e.,
\begin{equation}\label{eq:index2}
\begin{aligned}
\begin{bmatrix}E_{11} & 0 \\E_{21} & 0 \end{bmatrix}\begin{bmatrix}
\dot x_1 \\ \dot x_2
\end{bmatrix}&=\begin{bmatrix}
L_{11}&L_{12}  \\ L_{21}&L_{22}
\end{bmatrix}\begin{bmatrix}
Q_{11} & Q_{12} \\Q_{21} & Q_{22}
\end{bmatrix}\begin{bmatrix}
x_1 \\ x_2
\end{bmatrix}+
\left(\begin{bmatrix}B_1\\B_2\end{bmatrix}-\begin{bmatrix}P_1\\P_2 \end{bmatrix}\right)u\\
y&=
\left(\begin{bmatrix}B_1\\B_2\end{bmatrix}+\begin{bmatrix}P_1\\P_2 \end{bmatrix}\right)^T
\begin{bmatrix}
 Q_{11} & Q_{12}  \\ Q_{21} & Q_{22}
\end{bmatrix}\begin{bmatrix}
x_1 \\ x_2
\end{bmatrix}+(S+N)u,
\end{aligned}
\end{equation}
This system \eqref{eq:index2} can then be decoupled as described in the previous paragraph.

\section{Model reduction techniques}\label{sec:3}
\setcounter{equation}{0}
\noindent
In this section we present two different classes of model order reduction methods for pHDAEs. The first class is based on the reduction of the underlying Dirac structure and the associated power conservation, whereas the second one, the well-known moment matching, aims at the approximation of the transfer function.  To derive the reduction methods, we assume that the  pHDAE is in analogous form as \eqref{DAE_Q_Zeros}, but for convenience we neglect the feed-through terms, i.e., we assume that $S=N=0$ and then as consequence $P=0$, since $W\geq 0$. However, generalizations to systems with feed-through term are straightforward.

In the following we use an adapted notation, considering the pHDAE
\begin{align}\label{pHDAEblock}\nonumber
\begin{bmatrix}I_{n_1}&0&0\\0&I_{n_2}&0\\0&0&0\end{bmatrix}\begin{bmatrix}\dot{x}_{1}\\\dot{x}_{2}\\\dot{x}_3
\end{bmatrix}&=\left(\begin{bmatrix}J_{11} & J_{12} & J_{13}\\J_{21} & J_{22} & J_{23}\\ J_{31} & J_{32} & J_{33}
\end{bmatrix}-\begin{bmatrix}R_{11} & R_{12} & R_{13}\\R_{12}^T & R_{22} & R_{23}\\ R_{13}^T & R_{23}^T & R_{33}\end{bmatrix}\right)\begin{bmatrix}Q_{11} &0&0\\0&0&0\\Q_{31}&Q_{32}&Q_{33}\end{bmatrix}\begin{bmatrix}
x_{1}\\x_{2}\\x_3\end{bmatrix}\\
& \quad \nonumber +\begin{bmatrix}B_{1}\\ B_{2}\\ B_3
\end{bmatrix}u,\\
y&=\begin{bmatrix}B_{1}\\ B_{2}\\ B_3
\end{bmatrix}^T\begin{bmatrix}Q_{11} &0&0\\0&0&0\\Q_{31}&Q_{32}&Q_{33}\end{bmatrix}\begin{bmatrix}
x_{1}\\x_{2}\\x_3\end{bmatrix}.
\end{align}
for the states $x_i\in \mathbb{R}^{n_i}$, $\sum_{i=1}^3 n_i=n$, with the accordingly block-structured coefficient matrices $E$, $Q$, $Q^TJQ=-Q^TJ^TQ$, $Q^TRQ=Q^TR^TQ\in \mathbb{R}^{n\times n}$ and $B\in \mathbb{R}^{n\times p}$, where $(J_{j3}-R_{j3})=0$ for $j=1,2$,  $Q_{11}>0$ and $Q_{33}$ is nonsingular.

\subsection{Power conservation based model order reduction}\label{sec:power_method}
\noindent The power conservation based methods were originally developed for standard pHODEs in \cite{PolS12}, i.e., systems of the form \eqref{BEAMXZ18} with $E=I$ being the identity. If for such a system a splitting of the dynamic state as $x=[ x_r^T  \,\, x_s^T]^T$ exists, with $x_r\in\mathbb{R}^r$ and $x_s\in\mathbb{R}^{n-r}$, where $x_s$ does not contribute much  to the input-output behavior of the system, then the general idea is to cut the interconnection between the part of the energy storage port belonging to $x_s$ and the Dirac structure, such that no power is transferred. Then the power is exclusively exchanged via the energy storage of $x_r$, which will act as reduced state variable, whereas $x_s$ will be skipped. The constitutive relations become
\begin{align*}
\dot{x}_i=-f_{x_i}, \qquad \nabla_{x_i} H(x)=e_{x_i}, \quad i\in\{r,s\}
\end{align*}
with $e_x=[e_{x_r}^T\,\, e_{x_s}^T]^T$ and $f_x=[f_{x_r}^T\,\, f_{x_s}^T]^T$. To cut  the interconnection, one forces one of the power products
$(\nabla_{x_s} H)^T\dot{x}_s$ or $e_{x_s}^Tf_{x_s}$ to be zero.
Choosing  $\nabla_{x_s} H(x)=e_{x_s}=0$ yields the \emph{effort constraint reduction method (ECRM)}, while setting $\dot{x}_s=-f_{x_s}=0$ results in the \emph{flow constraint reduction method (FCRM)}.

We now adopt these reduction procedures of \cite{PolS12} to reduce pHDAEs in the form \eqref{DAErep}.
Proceeding from \eqref{pHDAEblock}, let $\hat{V}^{-1}x_1=[x_{1,r}^T\, \, x_{1,s}^T]^T$, $x_{1,r}\in \mathbb{R}^r$,  $x_{1,s}\in \mathbb{R}^{n_1-r}$, be an appropriate splitting of the dynamic part of the state variable with respect to its relevance for the input-output behavior. In the resulting system transformed by means of $V^{-T}$ and $V$ as in Lemma~\ref{TransfoTheo} with a block diagonal matrix $V=\mathrm{diag}(\hat V, I_{n_2} , I_{n_3})\in \mathbb{R}^{n\times n}$, we denote the state by $x=[{x^T_{1,r}} \,\, {x^T_{1,s}} \,\, x_2^T  \,\, x_3^T]^T$. By definition, the flow and effort variables of the energy-storing and energy-dissipating ports inherit the partitioning, and the coefficient matrices are structured accordingly.  The constitutive relations then read $f_x=-(V^{-1}EV)\dot{x}$ and $e_x=(V^TQV)x$, with
\begin{align}\label{eq:fe}
\begin{bmatrix}
f_{x_{1,r}}\\f_{x_{1,s}}\\f_{x_2}\\f_{x_3}
\end{bmatrix} =-
\begin{bmatrix}
I& 0& 0&  0\\
0& I& 0&  0\\
0&  0& I&  0\\
0&  0& 0& 0
\end{bmatrix}
\begin{bmatrix}
\dot{x}_{1,r}\\\dot{x}_{1,s}\\\dot{x}_2\\\dot{x}_3
\end{bmatrix},
\qquad
\begin{bmatrix}
e_{x_{1,r}}\\e_{x_{1,s}}\\e_{x_2}\\e_{x_3}
\end{bmatrix} =
\begin{bmatrix}
Q_{11}^{rr} & Q_{11}^{rs} & 0 & 0\\
(Q_{11}^{rs})^T& Q_{11}^{ss} & 0 & 0\\
0 & 0 & 0 & 0\\
Q_{31}^r & Q_{31}^s & Q_{32} & Q_{33}
\end{bmatrix}
\begin{bmatrix}
x_{1,r} \\ x_{1,s} \\ x_2 \\ x_3
\end{bmatrix}.
\end{align}
For the model reduction we have to open the resistive port. The transformed symmetric positive semi-definite dissipation matrix $V^{-1}RV^{-T}\in \mathbb{R}^{n\times n}$ admits  an ordered Schur decomposition
\begin{align}\label{LowRankR}
V^{-1}RV^{-T}=\begin{bmatrix}C&\hat{C}\end{bmatrix}\begin{bmatrix}\hat{R}&0\\0&0\end{bmatrix}\begin{bmatrix}
C^T\\\hat{C}^T\end{bmatrix}=C\hat{R}C^T,
\end{align}
with $0<\hat{R}=\hat {R}^T\in\mathbb{R}^{\ell\times \ell}$ and $C\in\mathbb{R}^{n\times \ell}$, with $\ell$ being the number of energy-dissipating elements. Plugging \eqref{LowRankR} into the transformed system and introducing the associated flow and effort variables accordingly, i.e., $f_R=-\hat{R}e_R$ and $e_R=C^T(V^TQV)x=C^Te_x$, yields a pHDAE with opened resistive port. Inserting the constitutive relations \eqref{eq:fe} and introducing the external port variables $(f_P,e_P)=(y,u)$, where $y=(V^{-1}B)^T (V^TQV)x=(V^{-1}B)^Te_x$, we obtain a new representation of \eqref{pHDAEblock} as
\begin{equation}\label{WholeDAE}
\begin{aligned}
-\begin{bmatrix}
I_{r}&0& 0&  0\\
0& I_{n_1-r}& 0& 0\\
 0&  0& {I_{n_2}}&  0\\
 0&  0& 0& I_{n_3}\\
0&0& 0& 0\\
0&0& 0& 0\end{bmatrix}
\begin{bmatrix}
f_{x_{1,r}}\\f_{x_{1,s}}\\f_{x_2}\\f_{x_3}
\end{bmatrix}&=
\begin{bmatrix}
J_{11}^{rr}&J_{11}^{rs}& {J_{12}^r}&{J_{13}^r}\\
J_{11}^{sr}& J_{11}^{ss}& {J_{12}^s}& {J_{13}^s}\\
{J_{21}^r}& {J_{21}^s}&{J_{22}}&{J_{23}}\\
{J_{31}^r}& {J_{31}^s}&{J_{32}}&{J_{33}}\\
-(B_1^{r})^T&-(B_1^{s})^T&{-B_2^T}&{-B_3^T}\\
-(C_1^{r})^T&-(C_1^{s})^T&{-C_2^T}&{-C_3^T}\end{bmatrix}
\begin{bmatrix}e_{x_{1,r}}\\e_{x_{1,s}}\\e_{x_2}\\e_{x_3}
\end{bmatrix}\\
&\quad+\begin{bmatrix}C_1^{r}\\ C_1^{s}\\{C_2}\\{
C_3}\\0\\0\end{bmatrix}f_R
+\begin{bmatrix}
0\\0\\{0}\\{0}\\0\\I_{\ell}\end{bmatrix}e_R+\begin{bmatrix}
0\\0\\{0}\\{0}\\I_m\\0\end{bmatrix}y+
\begin{bmatrix}B_1^{r}\\
B_1^{s}\\{B_2}\\{B_3}\\0\\0\end{bmatrix}u.
\end{aligned}
\end{equation}

In the \emph{Effort Constraint Reduction Method (ECRM)} the energy transfer between the energy-storing elements and the Dirac structure is cut by setting $e_{x_{1,s}}=0$. Here, the Hamiltonian is considered only weakly influenced by $x_{1,s}$. The relation $x_{1,s}=-(Q_{11}^{ss})^{-1}(Q_{11}^{rs})^Tx_{1,r}$ follows in a straightforward way from \eqref{eq:fe} as $Q_{11}^{ss}=(Q_{11}^{ss})^T>0$. The reduced Dirac structure is obtained by multiplying \eqref{WholeDAE} from the left with any matrix $D^{\mathrm{ec}}$ of maximal rank satisfying $D^{\mathrm{ec}}F_{x_{1,s}}=0$, such as,
\begin{align*}
D^{\mathrm{ec}}=
\begin{bmatrix}I_{r}&0&0&0&0&0\\
0&0&I_{n_2}&0&0&0\\
0&0&0&I_{n_3}&0&0\\
0&0&0&0&I_m&0\\
0&0&0&0&0&I_{\ell}
\end{bmatrix}.
\end{align*}
Rewriting the resulting system again as DAE and closing the resistive port yields the reduced model \eqref{ECRM}.
\begin{theorem}[Reduced model by ECRM]\label{thm:ercm}
Consider a pHDAE of the form \eqref{pHDAEblock} with its representation \eqref{WholeDAE}. Then the reduced model obtained by ECRM with reduced state $x^{\mathrm{ec}}=[x^T_{1,r}\,\,x^T_2\,\,x^T_3]^T\in \mathbb{R}^{(r+{n_2}+{n_3})}$, $r\ll n_1$ is a pHDAE given by
\begin{align}\label{ECRM}\nonumber
\underbrace{\begin{bmatrix}I_r&0&0\\0&I_{n_2}& 0\\0&0&0\end{bmatrix}}_{E^{\mathrm{ec}}}
\begin{bmatrix}\dot{x}_{1,r}\\\dot{x}_2\\\dot{x}_3\end{bmatrix}&=
\left(\underbrace{\begin{bmatrix}
J_{11}^{rr}&J_{12}^{r}&J_{13}^{r}\\
J_{21}^{r}&J_{22}&J_{23}\\
J_{31}^{r}&J_{32}&J_{33}
\end{bmatrix}}_{J^{\mathrm{ec}}}\right.\\& \quad\nonumber \left.-\underbrace{\begin{bmatrix}
R_{11}^{rr}&R_{12}^r&R_{13}^r\\ R_{21}^{r}&R_{22}&R_{23}\\
R_{31}^r&R_{32}&R_{33}
\end{bmatrix}}_{R^{\mathrm{ec}}}\right)\underbrace{\begin{bmatrix}
{\hat Q}_{11}&0&0\\0&0&0\\ {\hat Q}_{31}&Q_{32}&Q_{33}
\end{bmatrix}}_{Q^{\mathrm{ec}}}
\begin{bmatrix}x_{1,r}\\x_2\\x_3
\end{bmatrix}+
\underbrace{\begin{bmatrix}B_1^{r}\\B_2\\B_3
\end{bmatrix}}_{B^{\mathrm{ec}}}u,\\
y^{\mathrm{ec}}&=(B^{\mathrm{ec}})^T Q^{\mathrm{ec}} x^{\mathrm{ec}},
\end{align}
with $J^{\mathrm{ec}}=-(J^{\mathrm{ec}})^T$, $R^{\mathrm{ec}}=(R^{\mathrm{ec}})^T\geq 0$, where ${\hat Q}_{11}=Q_{11}^{rr}-Q_{11}^{rs}(Q_{11}^{ss})^{-1}(Q_{11}^{rs})^T$ and ${\hat Q}_{31}=Q_{31}^r-Q_{31}^s (Q_{11}^{ss})^{-1}(Q_{11}^{rs})^T$.
\end{theorem}

\begin{proof}
In the described reduction the port-Hamiltonian properties are inherited. The skew-symmetry of $J^{\mathrm{ec}}$ follows trivially, since
\begin{align*}
J^{\mathrm{ec}}=D_aV^{-1}JV^{-T}D^T_a=-(J^\mathrm{ec})^T, \qquad D_a=\begin{bmatrix}I_{r}&0&0&0\\0&0&I_{n_2}&0\\0&0&0&I_{n_3}\end{bmatrix}.
\end{align*}
The symmetry and positive semi-definiteness 
\begin{align*}
(E^{\mathrm{ec}})^TQ^{\mathrm{ec}}=\begin{bmatrix}{\hat Q}_{11}&0\\0&0\end{bmatrix}=(Q^{\mathrm{ec}})^TE^{\mathrm{ec}} \geq0,
\end{align*}
follows from ${\hat Q}_{11}={\hat Q}_{11}^T>0$ as it is constructed from the Schur complement of a positive definite matrix.
Finally,
\begin{align*}
(Q^{\mathrm{ec}})^TR^{\mathrm{ec}}Q^{\mathrm{ec}}=D_bV^TQ^TRQVD^T_b\geq 0, \qquad
D_b=\begin{bmatrix}I_{r}&-Q_{11}^{rs}(Q_{11}^{ss})^{-1}&0&0\\
0&0&I_{n_2}&0\\0&0&0&I_{n_3}\end{bmatrix},
\end{align*}
where the component $x_{1,s}$ is projected out by means of $D_b$.
\end{proof}
In the \emph{Flow Constraint Reduction Method (FCRM)}, the energy transfer between the energy-storing elements and the Dirac structure is cut by setting $f_{x_{1,s}}=0$ and $\dot{x}_{1,s}=0$. Thus, $x_{1,s}$ is constant and can particularly be chosen as $x_{1,s}=0$. The reduced Dirac structure is obtained by multiplying \eqref{WholeDAE} from the left with any matrix $D^{\mathrm{fc}}$ of maximal rank satisfying $D^{\mathrm{fc}}E_{x_{1,s}}=0$, e.g., with
\begin{align*}
D^{\mathrm{fc}}=
\begin{bmatrix}
I_{r}& -J_{11}^{rs}(J_{11}^{ss})^{-1}&0&0&0&0\\
							0&-J_{21}^{s}\,(J_{11}^{ss})^{-1}&I_{n_2}&0&0&0\\
							0&-J_{31}^{s}\,(J_{11}^{ss})^{-1}&0&I_{n_3}&0&0\\
							0&\,\,\,(B_1^{s})^T\,\,(J_{11}^{ss})^{-1}&0&0&I_m&0\\						0&\,\,\,(C_1^{s})^T\,\,(J_{11}^{ss})^{-1}&0&0&0&I_{\ell}
\end{bmatrix},
\end{align*}
in the case that $(J_{11}^{ss})^{-1}$ exists.
Rewriting the system as a DAE, analogously to ECRM, yields the reduced model \eqref{FCRM}.
\begin{theorem}[Reduced model by FCRM]\label{thm:frcm}
Consider a pHDAE of the form \eqref{pHDAEblock} with its general representation \eqref{WholeDAE} and suppose that $J_{11}^{ss}$ in \eqref{WholeDAE} is invertible. Then the reduced order model obtained by FCRM is port-Hamiltonian and, for state $x^{\mathrm{fc}}=[x^T_{1,r}\,\,x^T_2\,\,x^T_3]^T\in \mathbb{R}^{(r+{n_2}+{n_3})}$, $r\ll n_1$, given  by
\begin{equation}\label{FCRM}
\begin{aligned}
\underbrace{\begin{bmatrix}I_r&0&0\\0&I_{n_2}&0\\
0&0&0\end{bmatrix}}_{E^\mathrm{fc}}
\begin{bmatrix}\dot{x}_{1,r}\\\dot{x}_2\\\dot{x}_3\end{bmatrix}
&=\left(\underbrace{\mathcal{J}-{\mathcal C}^TZ_J{\mathcal C}}_{J^\mathrm{fc}}-\underbrace{{\mathcal C}^TZ_R{\mathcal C}}_{R^\mathrm{fc}}\right)
\underbrace{\begin{bmatrix}Q_{11}^{rr}&0&0\\0&0&0\\
Q_{31}^r&Q_{32}&Q_{33}\end{bmatrix}}_{Q^\mathrm{fc}}
\begin{bmatrix}
x_{1,r}\\x_2\\x_3\end{bmatrix}\\
&\quad +\left(\underbrace{-{\mathcal B}^T-{\mathcal C}^TZ_J{\mathcal G}}_{B^\mathrm{fc}}-
\underbrace{{\mathcal C}^TZ_R{\mathcal G}}_{P^\mathrm{fc}}\right)u,\\
y^\mathrm{fc}&=\left(B^\mathrm{fc}+P^\mathrm{fc}\right)^T
Q^\mathrm{fc}x^\mathrm{fc}+\left(\underbrace{{\mathcal G}^T Z_R{\mathcal G}}_{S^\mathrm{fc}}+\underbrace{\mathcal G^T Z_J{\mathcal G}-{\mathcal N}}_{N^\mathrm{fc}}\right)u,
\end{aligned}
\end{equation}
with
\begin{align*}
\mathcal{J}&=[\mathcal{J}]_{k,l=1,2,3}=-\mathcal{J}^T, \\
\mathcal{J}_{11}&=J_{11}^{rr}-J_{11}^{rs}(J_{11}^{ss})^{-1}J_{11}^{sr},\,
&\mathcal{J}_{1j}&=J_{1j}^r-J_{11}^{rs} (J_{11}^{ss})^{-1}J_{1j}^{s}, &&  j=2,3,\\
\mathcal{J}_{2j}&=J_{2j}-J_{21}^{s} (J_{11}^{ss})^{-1}J_{1j}^{s}, \,
&\mathcal{J}_{33}&= J_{33}-J_{31}^s(J_{11}^{ss})^{-1}J_{13}^s, && j=2,3,\\
{\mathcal B}&=[{\mathcal B}]_{k=1,2,3},\\
{\mathcal B}_1&= (B_1^{s})^T(J_{11}^{ss})^{-1}J_{11}^{sr}-(B_1^{r})^T,\,
&{\mathcal B}_j&= (B_1^{s})^T(J_{11}^{ss})^{-1}J_{1j}^{s}-B_j^T,&& j=2,3,\\
{\mathcal C}&=[{\mathcal C}]_{k=1,2,3}, \\
{\mathcal C}_1&=(C_1^{s})^T(J_{11}^{ss})^{-1}J_{11}^{sr}-(C_1^{r})^T,\,
&{\mathcal C}_j&= (C_1^{s})^T(J_{11}^{ss})^{-1}J_{1j}^{s}-C_{j}^T,&& j=2,3,\\
\mathcal{G}&= (C_1^{s})^T(J_{11}^{ss})^{-1}B_1^{s},\,
&{\mathcal D}&= (C_1^{s})^T(J_{11}^{ss})^{-1}C_1^{s},
&& {\mathcal N}=(B_1^{s})^T(J_{11}^{ss})^{-1}B_1^{s},
\end{align*}
as well as $Z=\hat{R}(I-\mathcal{D}\hat{R})^{-1}$ with its symmetric and skew-symmetric parts,
$Z_R=(Z+Z^T)/2$ and $Z_J=(Z-Z^T)/2$.
\end{theorem}
\begin{proof}
In the described reduction the port-Hamiltonian properties are inherited.
It follows trivially that $J^\mathrm{fc}$, $N^\mathrm{fc}$ are skew-symmetric and $R^\mathrm{fc}$, $S^\mathrm{fc}$ are symmetric by construction. From $Q_{11}^{rr}=(Q_{11}^{rr})^T>0$ it results that $(E^\mathrm{fc})^TQ^\mathrm{fc}=(Q^\mathrm{fc})^TE^\mathrm{fc}\geq 0$. Finally,
\begin{align*}
\begin{bmatrix}(Q^\mathrm{fc})^TR^\mathrm{fc}Q^\mathrm{fc} & (Q^\mathrm{fc})^TP^\mathrm{fc}\\(P^\mathrm{fc})^TQ^\mathrm{fc}& S^\mathrm{fc}\end{bmatrix}=\begin{bmatrix}(Q^\mathrm{fc})^T&0\\0&I\end{bmatrix}
\begin{bmatrix}{\mathcal C}^T\\ {\mathcal G}^T\end{bmatrix}Z_R\begin{bmatrix}{\mathcal C} &{\mathcal G}\end{bmatrix}
\begin{bmatrix}Q^\mathrm{fc}&0\\0&I\end{bmatrix}\geq 0
\end{align*}
holds, since $Z_R$ is positive (semi-)definite, see \cite[]{PolS12}.
\end{proof}
The reduced models obtained by ECRM and FCRM have similarities but also show crucial differences. Obviously, $E^\mathrm{ec}=E^\mathrm{fc}$, implying the same size of the reduced states. The energy matrices $Q^\mathrm{ec}$, $Q^\mathrm{fc}$ only differ in the first column. In particular, ECRM generates an additive term in the matrix $\hat Q_{11}$, and also in the matrix $\hat Q_{31}$ associated to the algebraic constraints due to the elimination of $x_{1,s}$, cf.\ \eqref{ECRM}. Clearly, the models differ in the feed-through term which is only present in FCRM (see $S^\mathrm{fc}$, $N^\mathrm{fc}$ and $P^\mathrm{fc}$ in \eqref{FCRM}), even though the original system \eqref{pHDAEblock} did not have such a term. It is further important to note that the construction of ECRM is always applicable, whereas in the presented form FCRM requires the skew-symmetric matrix $J^{ss}_{11}$ to be invertible, which is, e.g., impossible if the size $(n_1-r)$ is odd. In the case that $J^{ss}_{11}$ is singular, the procedure has to be modified, but we do not present this modification here because it gets rather technical.
\begin{remark}\label{rem:diff}
In the presented power conservation based methods the general representation \eqref{WholeDAE} for the pHDAE differs from the one for a standard pHODE \cite{PolS12} by the two additional block rows and columns for the equations associated with the kernel of the energy matrix and with the algebraic constraints. Hence, ECRM yields a reduced model \eqref{ECRM} with equivalent block matrices in the dynamic part. However, although the reduction is only applied to the dynamic state, additional terms in $\hat Q_{31}$ associated to the algebraic constraints are generated. Also the reduced model \eqref{FCRM} by FCRM contains additional blocks, i.e., $\mathcal{J}_{1j}$, $\mathcal{J}_{2j}$, $\mathcal{J}_{33}$, ${\mathcal B}_j$ and ${\mathcal C}_j$, $j\in\{2,3\}$, which arise from the additional equations. The coefficients $\mathcal{J}_{11}$, ${\mathcal B}_1$, ${\mathcal C}_1$, $\mathcal G$, $\mathcal D$ and $\mathcal N$ are analogous to their counterparts for a pHODE.
\end{remark}

\subsection{Moment Matching}
\noindent The model reduction procedure of moment matching (MM) derives a reduced order model by means of a Galerkin projection in such a way that the leading coefficients of the series expansion of its transfer function (its moments) match those of the full order system, i.e.,\ $G(s)=B^T(s E-(J-R)Q)^{-1}B=\sum_{j=0}^\infty m_j(s_0-s)^j$ with moments $m_j$ associated with a given shift parameter $s_0$. For details of MM for DAEs we refer to \cite{Fre05} for $s_0\in\mathbb{C}$ and to \cite{BenS06} for $s_0=\infty$.  To apply these techniques in a structure-preserving way to the pHDAE \eqref{pHDAEblock}, the symmetric positive definite energy matrix block $Q_{11}$ associated to the dynamic state is first transformed to become an identity, as done in the works on MM for pHODEs \cite{PolS10, PolS11}. Performing a Cholesky factorization $Q_{11}=K K^T$ and transforming appropriately, by Lemma~\ref{TransfoTheo} with $U=\mathrm{diag}(K, I_{n_2},I_{n_3})$, $V=U^{-T}\in \mathbb{R}^{n\times n}$, the resulting system is still port-Hamiltonian. Then a Galerkin projection matrix for the dynamic part $V_r\in \mathbb{R}^{n_1\times r}$, $r\ll n_1$, can be computed, e.g., by an Arnoldi method \cite{Fre05}, such that $V_r^TV_r=I_r$ and its columns span a Krylov space of associated to the system shifted by $s_0$ \cite{Saa03}. Applying finally the Galerkin projection with $V^\mathrm{m}_r=\mathrm{diag}(V_r, I_{n_2},I_{n_3})\in \mathbb{R}^{n\times(r+n_2+n_3)}$ yields the reduced model.
\begin{theorem}[Reduced model by MM]
Consider a pHDAE of the form \eqref{pHDAEblock} with energy matrix block $Q_{11}=K K^T$. Let the projection matrix $V_r\in \mathbb{R}^{n_1\times r}$ be computed by the Arnoldi method. Then the reduced system is port-Hamiltonian, matches the first $r$ moments of the full order system and is for the state $x^\mathrm{m}=[x_{1,r}^T\,\,x_2^T\,\,x_3^T]^T\in \mathbb{R}^{(r+n_2+n_3)}$, $r\ll n_1$,
with $x_{1,r} = V_r^{T} K^{T} x_1$
given by
\begin{align*}\label{MM}
&\underbrace{\begin{bmatrix}I_r&0&0\\0&I_{n_2}& 0\\0&0&0\end{bmatrix}}_{E^{\mathrm{m}}}
\begin{bmatrix}\dot{x}_{1,r}\\\dot{x}_2\\\dot{x}_3\end{bmatrix}\\ \nonumber &  \quad \quad =
\underbrace{\begin{bmatrix}
V^T_rK^{T}L_{11}K V_r &V^T_rK^{T} L_{{12}}& V^T_rK^{T} L_{{13}}\\
L_{21}K V_r&L_{22}&L_{23}\\
L_{31}K V_r&L_{32}&L_{33}
\end{bmatrix}}_{J^{\mathrm{m}}-R^\mathrm{m}}\underbrace{\begin{bmatrix}
I_r&0&0\\0&0&0\\Q_{31}K^{-T}V_r&Q_{32}&Q_{33}
\end{bmatrix}}_{Q^{\mathrm{m}}}
\begin{bmatrix}x_{1,r}\\x_2\\x_3
\end{bmatrix}\\ & \quad \quad \quad +
\underbrace{\begin{bmatrix}V^T_rK^{T} B_{1}\\B_2\\B_3
\end{bmatrix}}_{B^{\mathrm{m}}}u,\\
& \quad y^{\mathrm{m}}=(B^{\mathrm{m}})^T Q^{\mathrm{m}} x^{\mathrm{m}}
\end{align*}
with $L_{ij}=J_{ij}-R_{ij}$, $i,j=1,2,3$.
\end{theorem}
\begin{proof}
The port-Hamiltonian structure is trivially preserved by the $Q_{11}$-associated transformation and the subsequent Galerkin projection.. The matching of the first $r$ moments is proved in \cite{Fre05} for $s_0\in\mathbb{C}$ and in \cite{BenS06} for $s_0\in\infty$.
\end{proof}

\section{Numerical results}\label{sec:4}
\setcounter{equation}{0}
\noindent
In this section we investigate the performance of the model reduction methods, using benchmark examples from the literature, see, e.g., \cite{BeaGM17, BorG15, EggKLMM18, GugSW13, HeiSS08, MehS05} or \cite{Sty06}.  In order to perform model reduction, it is essential to identify all constraints arising from the physics of the problem as discussed in Section~\ref{sec:decoupling}. In many applications this can be done directly by exploiting the structure of the equations coming from the physical properties.

Considering the transfer function $G$ of the original full order system, we study the approximation quality of the various reduced order systems by comparing the relative errors $(G-G_r)/G$ with $G_r\in\{G^\mathrm{ec},G^\mathrm{fc}, G^\mathrm{m}\}$ being the transfer function of the reduced system. Usually this is done in the $L_2$-norm in the state-space formulation or in the $H_\infty$-norm in the frequency domain. Since the latter norm is only defined for pHDAEs of index at most one, for pHDAEs of higher index we present the errors for the index reduced system. The transfer function of a pHDAE of the form \eqref{BEAMXZ18} is  given by
\begin{align*}
{G}(s)= (B+P)^T(s E-(J-R)Q)^{-1}(B-P)+(S+N), \quad s\in\mathbb{C}.
\end{align*}
For pHDAEs of index at most one it is either a proper rational function or the sum of a proper rational function with a term that is constant in $s$.

The numerical results have been computed with Matlab 2017a on a Linux 64-Bit machine with an Intel\textsuperscript{\textregistered} Core\textsuperscript{\texttrademark} i7-6700 processor.
In the context of computing the error norms, see \cite{AliBMSV17,Sty06b}, it is necessary to solve
Lyapunov equations, for which we have used the M.M.E.S.S. Toolbox \cite{SakKB16}.

\subsection{Flow problems}
\noindent Consider an instationary incompressible fluid flow, prescribed in terms of velocity $v:\Omega\times [0,T]\rightarrow \mathbb{R}^2$ and pressure $p:\Omega\times [0,T]\rightarrow \mathbb{R}$ on the spatial domain $\Omega=(0,1)^2$ with boundary $\partial \Omega$ for the time period $[0,T]$, that is driven by external forces $f:\Omega\times [0,T]\rightarrow \mathbb{R}^2$ and dynamic viscosity $\nu>0$. The subsequently described models of partial differential equations (Stokes as well as Oseen equations) are closed by non-slip boundary conditions and appropriate initial conditions $v^0$. Spatial discretization by a finite difference method on a uniform staggered grid yield pHDAEs for the state $x=[ v_h^T \,\, p_h^T]^T$ with the semi-discretized values of the velocities $v_h(t)\in\mathbb{R}^{n_v}$ and pressures $p_h(t)\in\mathbb{R}^{n_p}$, $t\in [0,T]$ ($n_v, n_p\in \mathbb{N}$, cf.\ Remark~\ref{rem:discretization}).

\subsubsection*{Stokes equations}
A laminar flow can be modeled by the linear \emph{Stokes equations},
\begin{align*}
\partial_t v&= \nu\Delta v -\nabla p+f, &&\text{in }\Omega\times(0,T],\\
0&= - \operatorname{div}v,  &&\text{in }\Omega\times(0,T],\\
v&=0,  &&\text{on } \partial\Omega\times(0,T],\\
v&=v^0, &&\text{in }\Omega \times \{0\}.
\end{align*}
A spatially semi-discretized differential-algebraic system for $x=[ v_h^T \,\, p_h^T]^T$ completed with an appropriate output equation is given by
\begin{equation}\label{StokesDAE}
\begin{aligned}
\underbrace{\begin{bmatrix}I & 0\\0&0\end{bmatrix}}_E\begin{bmatrix}\dot v_h \\ \dot p_h \end{bmatrix}&=\Bigg(\underbrace{\begin{bmatrix}0 & -D^T\\D&0\end{bmatrix}}_J-\underbrace{\begin{bmatrix}-\mathcal{L}&0\\0&0\end{bmatrix}}_R\Bigg)
\begin{bmatrix}v_h \\  p_h \end{bmatrix} +\underbrace{\begin{bmatrix}F\\0\end{bmatrix}}_Bu,\\
y&=B^T x,
\end{aligned}
\end{equation}
with the symmetric negative definite discrete Laplace operator $\mathcal{L}\in\mathbb{R}^{n_v\times n_v}$, as well as the discrete divergence $D$ and gradient operators $D^T\in\mathbb{R}^{n_v\times n_p}$, $n_p<n_v$.  The operator $D$ usually has full row rank if the freedom in the pressure (which only occurs in differentiated form in the system) is removed, see Remark~\ref{rem:discretization} on page~\pageref{rem:discretization}. The initial conditions are $v_h(0)=v_{h}^0$ and consistently $p_h(0)=p_{h}^0$. The input $u$ with input matrix $F\in\mathbb{R}^{n_v\times m}$ results from the external forces. The output equation is supplemented accordingly regarding the port-Hamiltonian form \eqref{BEAMXZ18}. System~\eqref{StokesDAE} is obviously port-Hamiltonian with $Q=I$, $P=0$ and $S=N=0$, as $E^TQ=Q^TE\geq 0$, $J=-J^T$ and $R=R^T\geq 0$ holds.

For the index reduction of the pHDAE (\ref{StokesDAE}) (which is of differentiation-index two) we do not need the whole derivative array; instead we can easily identify the equations that have to be differentiated from the special structure of the system by performing, e.g., a singular value decomposition
\begin{align*}
D^T=U \begin{bmatrix}\Sigma_D & 0 \end{bmatrix}^TV^T, \qquad
\Sigma_D=\mathrm{diag}(\sigma_1,\dots,\sigma_{n_p})\in\mathbb{R}^{n_p\times n_p},
\end{align*}
with orthogonal matrices $U\in\mathbb{R}^{n_v\times n_v}$, $V\in\mathbb{R}^{n_p\times n_p}$ and singular values $\sigma_i> 0$, $i=1,\dots,n_p$. Setting ${Z}=V\Sigma_D$, performing an equivalence transformation with $U$ and splitting the state variable accordingly into three parts, we get the system
\begin{align*}
\begin{bmatrix} I&0&0\\0&I&0\\0&0&0\end{bmatrix}\begin{bmatrix}\dot{x}_1\\\dot{x}_2\\\dot{x}_3\end{bmatrix}
&=\left(\begin{bmatrix}0&0&-{Z}^T\\0&0&0\\{Z}&0&0\end{bmatrix}-\begin{bmatrix}-\mathcal{L}_{11}&-\mathcal{L}_{12}&0\\-\mathcal{L}_{12}^T&-\mathcal{L}_{22}&0\\0&0&0\end{bmatrix}\right)\begin{bmatrix}x_1\\x_2\\x_3\end{bmatrix}+\begin{bmatrix}
B_1\\B_2\\0\end{bmatrix}u,\\
y&=\begin{bmatrix}B_1^T&B_2^T&0\end{bmatrix}
\begin{bmatrix}x_1\\x_2\\x_3\end{bmatrix}.
\end{align*}
Obviously $x_1=0$, as the last equation is $0={Z}x_1$ with $Z$ invertible. This is the equation that has to be differentiated and inserted into the first equation to derive the second (hidden) algebraic constraint
\begin{align*}
Z^Tx_3=\mathcal{L}_{12}x_2+B_1u
\end{align*}
as well as a consistency condition for the initial value which relates the initial condition for $u$ and $x_2$ to that for $x_3$.
The second equation yields the underlying ODE of the system for the variable $x_2:[0,T]\rightarrow \mathbb{R}^{(n_v-n_p)}$ to be reduced and the output equation
\begin{equation}\label{StokesUOde}
\begin{aligned}
\dot{x}_2&=\mathcal{L}_{22}x_2+B_2u,\\
y&=B_2^Tx_2.
\end{aligned}
\end{equation}
Note that this equation can be interpreted as the discretized heat equation  in the set of divergence-free velocities \cite{EmmM13}.

In system~\eqref{StokesUOde} the skew-symmetric interconnection matrix is zero. Thus, FCRM cannot be applied directly as discussed in the previous section. Concerning ECRM, the splitting of the dynamic state is provided by Lyapunov balancing. The transformation matrices are computed by the Square-Root Algorithm \cite{Ant05}.
Note that in this case the reduced model by ECRM is equivalent to the one obtained by using Balanced Truncation due to the symmetry of the system matrix and the relation of the input and output matrices.

\subsubsection*{Oseen equations}
A flow model that is closer to the nonlinear \emph{Navier-Stokes equations} is given by the
\emph{Oseen equations}
\begin{align*}
\partial_t v&= -(a\cdot \nabla) v + \nu\Delta v -\nabla p+f, &&\text{in }\Omega\times(0,T],\\
0&= - \operatorname{div}v,  &&\text{in }\Omega\times(0,T],\\
v&=0,  &&\text{on } \partial\Omega\times(0,T],\\
v&=v^0, &&\text{in }\Omega \times \{0\}.
\end{align*}
The Oseen equations differ from the Stokes equations by the additional convective term with driving velocity $a:\Omega\rightarrow \mathbb{R}^2$. The associated spatially semi-discretized pHDAE for $x=[ v_h^T \,\, p_h^T]^T$ is given by
\begin{align*}
\begin{bmatrix}I&0\\0&0\end{bmatrix}\dot{x}&=\left(\begin{bmatrix}{A}&-D^T\\D&0\end{bmatrix}-\begin{bmatrix}
-\tilde{\mathcal{L}}&0\\0&0\end{bmatrix}\right)x+\begin{bmatrix}F\\0\end{bmatrix}u, \\
y&=B^T x, \qquad   B^T=\begin{bmatrix} F^T & 0 \end{bmatrix},
\end{align*}
where the appropriately discretized convective term is decomposed into its skew-symmetric part $A$ and its symmetric part. The last forms, together with the discrete Laplacian, the symmetric negative definite operator $\tilde{\mathcal{L}}$.

Analogously to the index reduction performed for the Stokes equations, we obtain $x_1=0$ as well as $Z^Tx_3=(A_{12}+\mathcal{\tilde L}_{12})x_2+B_1u$. The underlying ODE and the output equation are given by
\begin{equation}\label{OseenUOde}
\begin{aligned}
\dot{x}_2&=(A_{22}-(-\tilde{\mathcal{L}}_{22}))x_2+B_2u,\\
y&=B_2^Tx_2.
\end{aligned}
\end{equation}

In system~\eqref{OseenUOde} the skew-symmetric interconnection matrix is prescribed by $A_{22}$ which, depending on the discretization scheme, may or may not be invertible. If it is invertible, then in contrast to the Stokes problem, FCRM can be applied for model order reduction.  For ECRM and FCRM the balancing transformations may be computed via the Balancing Free Square-Root  Algorithm \cite{Var91} to preserve the structure and to avoid creating an additional energy matrix.
\begin{remark}
\label{rem:discretization}{\rm
In the numerical examples the flow domain $\Omega=(0,1)^2$ is partitioned into uniform quadratic cells of edge length $h=1/M$, $M\in \mathbb{N}$. We use finite differences on a staggered grid where the velocity components $v=[v^\xi \,\, v^\eta]^T$ are evaluated at the center of the cell faces to which they are normal and the pressure is taken at the cell centers. This procedure provides small discretization stencils and ensures numerical stability. The unknowns are
\begin{align*}
v^\xi_{i,j+0.5}(t)&\approx v^\xi(ih,(j+0.5)h,t), &&i=1,...,M-1,\quad  j=0,....,M-1\\
v^\eta_{i+0.5,j}(t)&\approx v^\eta((i+0.5)h,jh,t), && i=0,...,M-1, \quad j=1,....,M-1\\
p_{i+0.5,j+0.5}(t)&\approx p((i+0.5)h,(j+0.5)h,t), && i,j=0,...,M-1,
\end{align*}
for $t\in[0,T]$. Since the pressure is non-unique in the flow equations, we fix without loss of generality the value $p_{M-0.5,M-0.5}(t)=0$ for the numerical solution and discard this quantity from the unknowns. The unknowns are ordered row-wise (and for the velocity component-wise) in the vectors $v_h(t)\in \mathbb{R}^{n_v}$ and $p_h(t)\in \mathbb{R}^{n_p}$, where $n_v= 2M(M-1)$ and $n_p=M^2-1$, yielding the state  $x=[v_h^T \,\, p_h^T]^T:[0,T]\rightarrow \mathbb{R}^n$, $n=3M^2-2M-1$, of the pHDAE systems.
}
\end{remark}

\begin{figure}[t]
\includegraphics[width=0.49\textwidth]{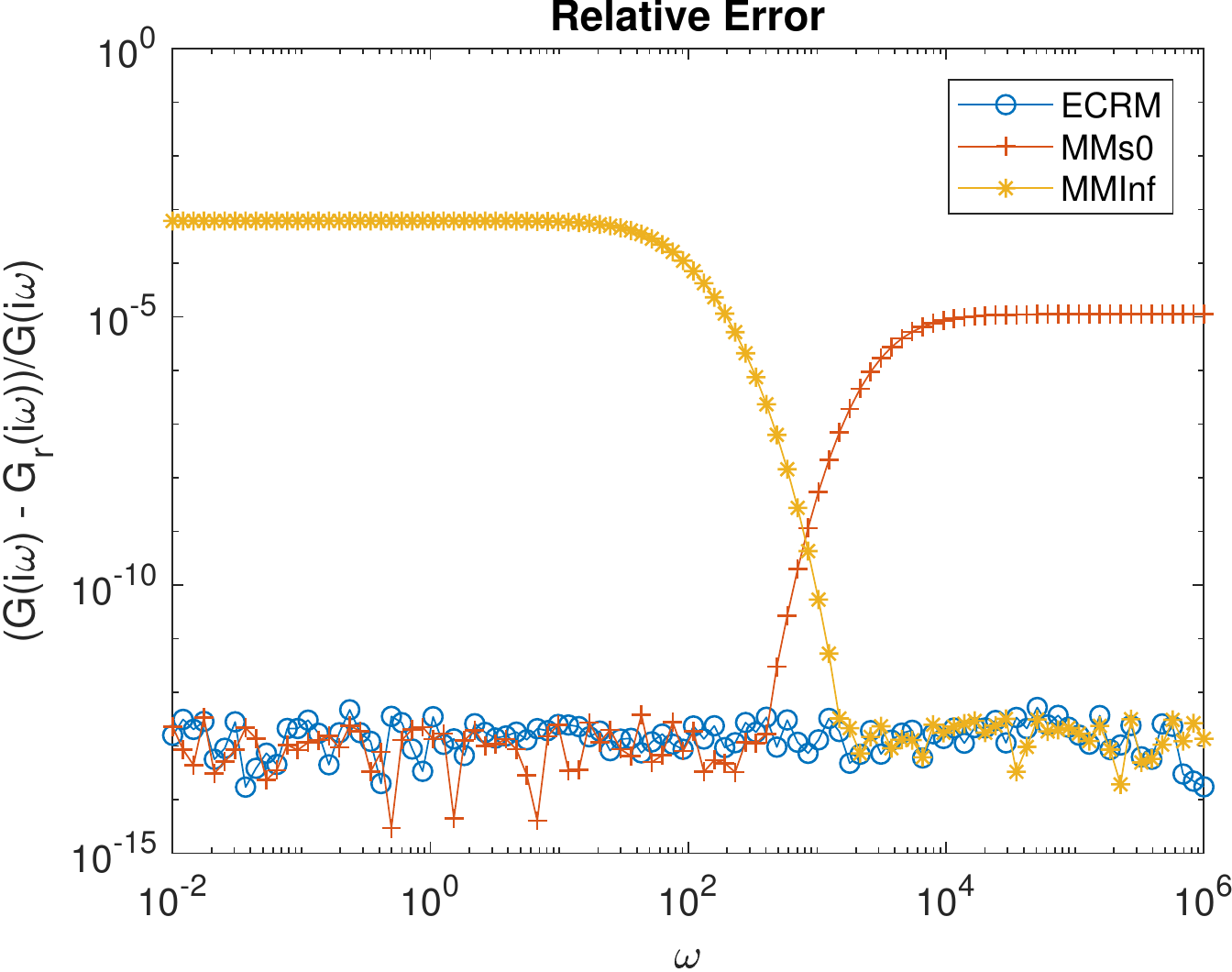}\hfill
\includegraphics[width=0.49\textwidth]{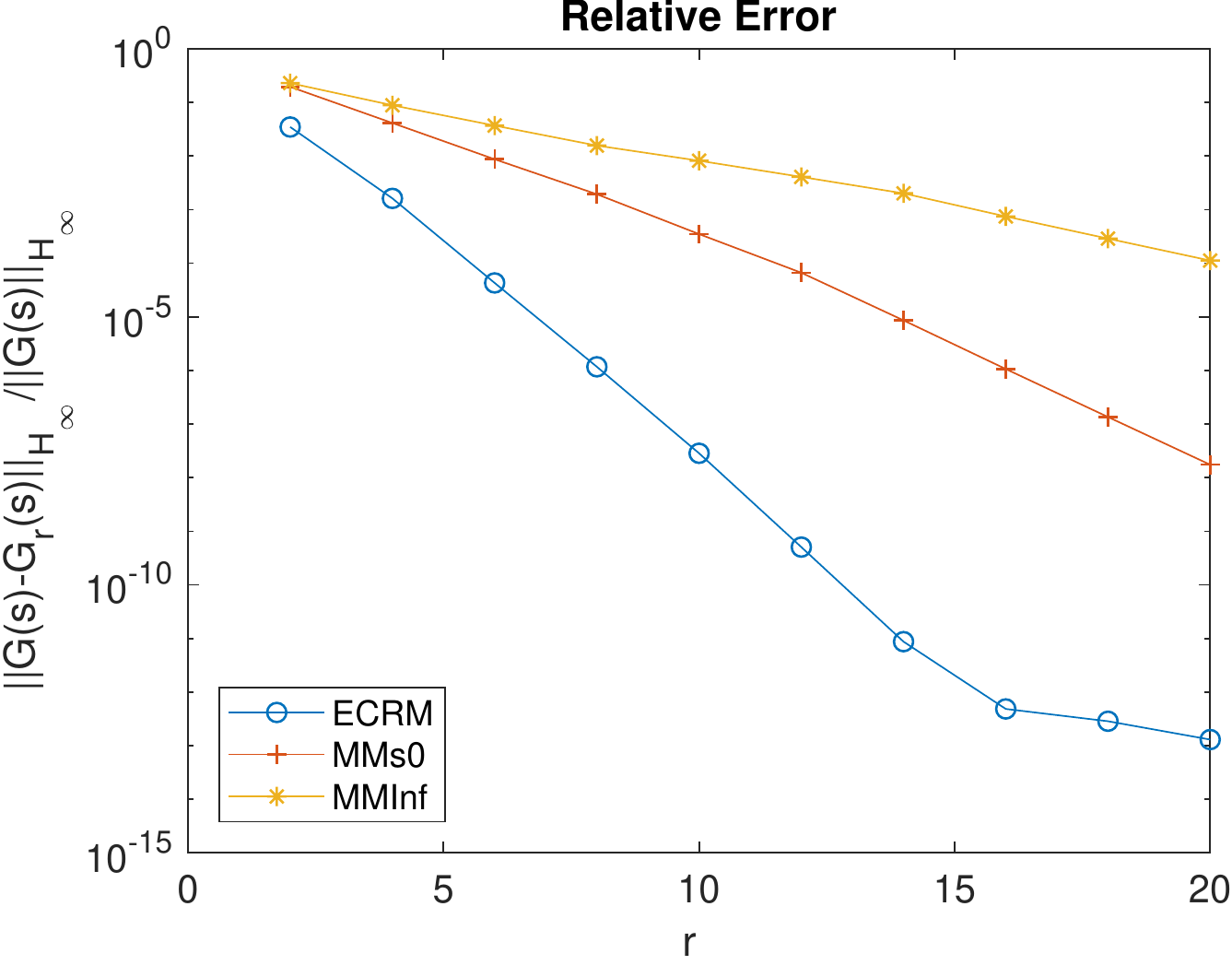}\\
\includegraphics[width=0.49\textwidth]{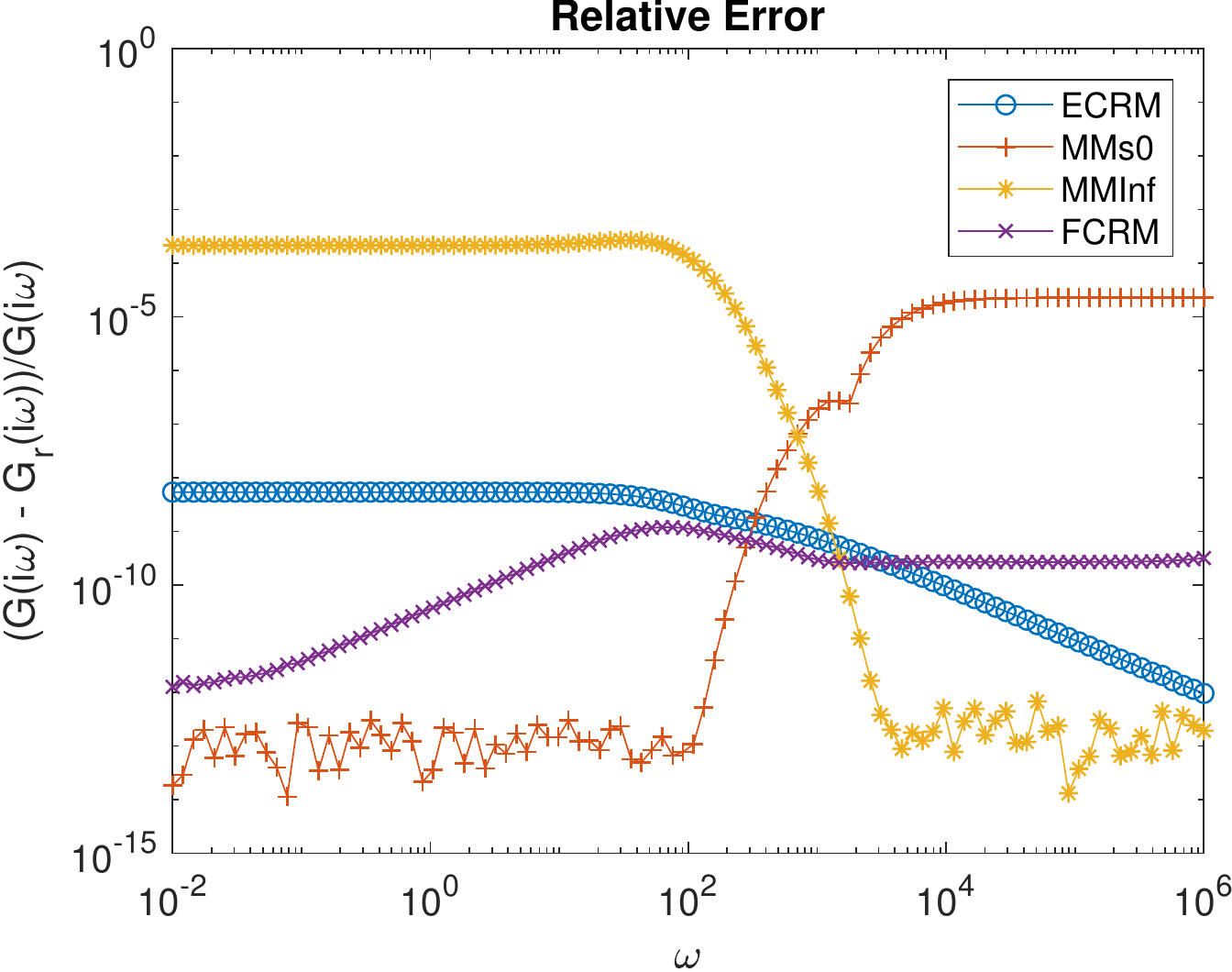}\hfill
\includegraphics[width=0.49\textwidth]{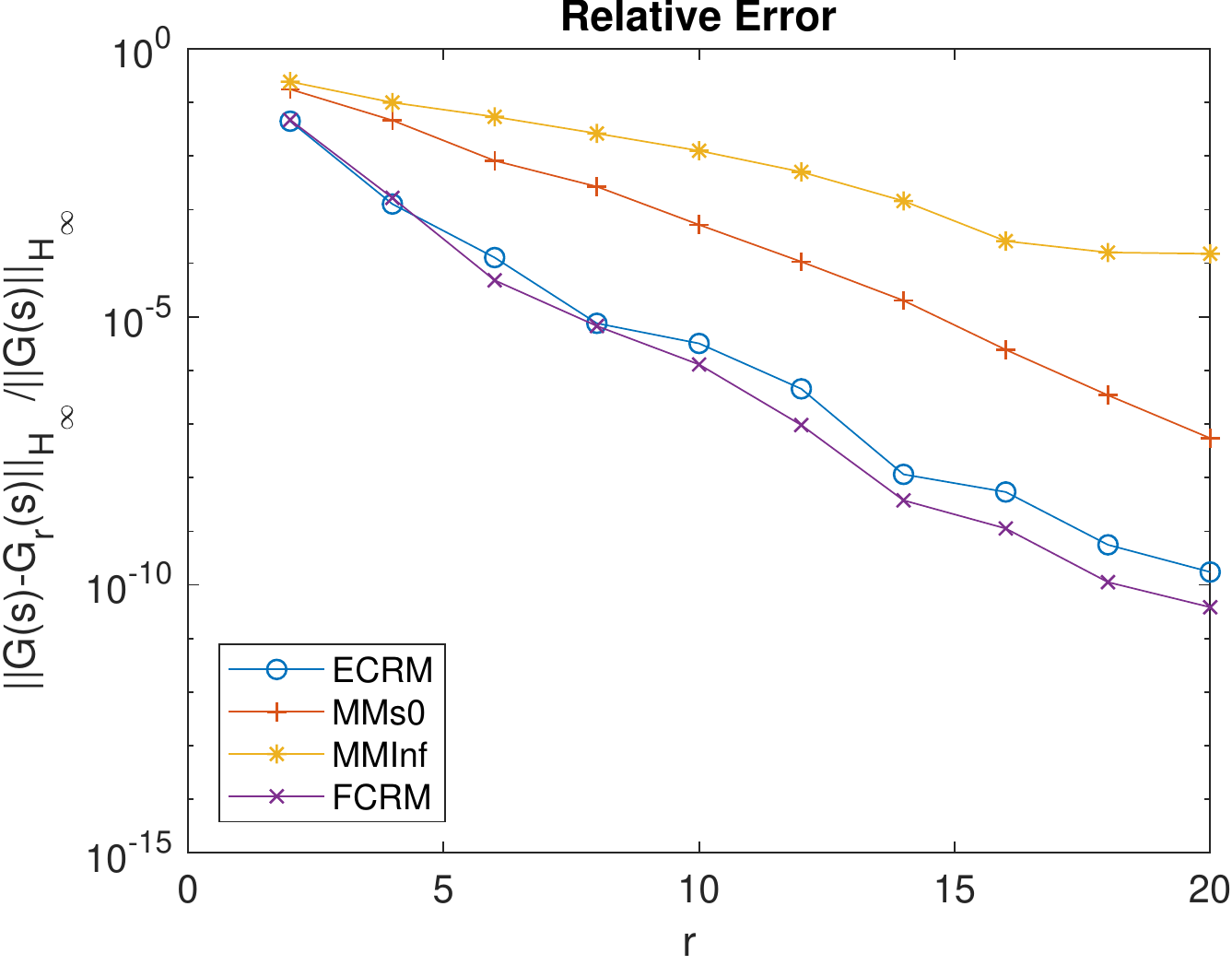}
\caption{Flow problems: Stokes \textit{(top)} and Oseen \textit{(bottom)}.
Relative errors for ECRM, FCRM as well as MM at $s_0=\infty$ and at $s_0=0$, plotted over frequency \textit{(left)} and in $H_\infty$-norm over reduced state size $r$ \textit{(right)}.}
\label{FlowROMs}
\end{figure}

We illustrate the model reduction techniques by comparing the approximation quality of their reduced transfer functions for both flow problems. The presented results are given for the example flow setup, where we use a random normally distributed input matrix $B\in (\mathcal{N}(0,10^{2}))^{n_v\times 1}$, the dynamic viscosity $\nu=1$ and, in case of the Oseen equations, the constant convective velocity $a\equiv[1\,\,1]^T$. Applying a spatial resolution of $M=23$, the states have a size of $n=1540$ for the full order pHDAE models and of $n_2=484$ for the underlying ODEs to be reduced. Figure~\ref{FlowROMs} shows the relative errors in the spectral norm for the reduced models of size $r=16$ and in the $H_\infty$-norm for $r\in[2,20]$. The results of MM are similar for the Stokes and Oseen equations. As expected, MM at $s_0=\infty$ and at $s_0=0$ yields only negligibly small errors for high or low frequencies, respectively. In the $H_\infty$-norm MM at $s_0=0$ performs better than MM at $s_0=\infty$.
\begin{figure}[ht]
\includegraphics[scale=0.5]{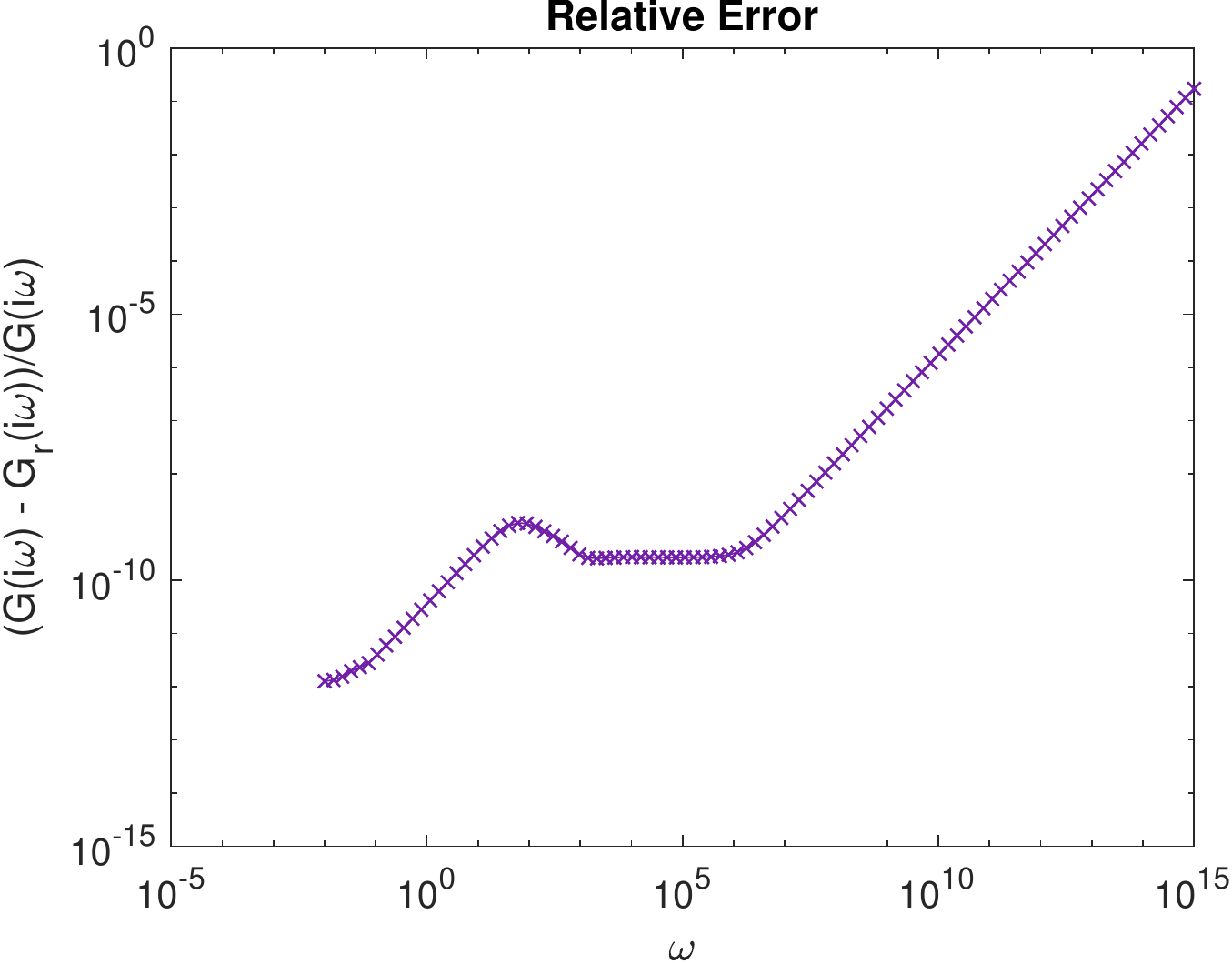}
\caption{Error behavior of FCRM for very high frequencies due to the additional feed-through term in the reduced Oseen model.}
\label{OseenROMs}
\end{figure}
For the Stokes problem the error of ECRM is small for all frequencies, oscillating around $\mathcal{O}(10^{-13})$, whereas for the Oseen problem it decreases from $\mathcal{O}(10^{-8})$ to $\mathcal{O}(10^{-13})$ for increasing frequencies. In the $H_\infty$-norm ECRM outperforms MM for both flow problems. The same error trends can be observed in the $H_2$-norm. Consequently, ECRM yields better reduced models than the moment matching methods globally, as even for low and high frequencies the relative errors only differ slightly from the respective errors of MM.
FCRM is not applicable to the Stokes problem and to the Oseen problem only if the invertibility requirement on the skew-symmetric interconnection submatrix $J^{ss}_{11}$ is satisfied. This implies in this example setup that the size $r$ of the reduced model has to be even. Here, FCRM yields smaller errors for low frequencies than ECRM, but for higher frequencies the error increases monotonically because of the additional feed-through term in the reduced system, see also Figure~\ref{OseenROMs}. However, in the $H_\infty$-norm FCRM even outperforms ECRM, whereas the $H_2$-norm is unbounded for systems with nonzero feed-through terms.

\subsection{Damped mass-spring system}
\noindent The holonomically constrained damped mass-spring system is a multibody problem that describes the one-dimensional dynamics of $g$ connected mass points in terms of their positions $p:[0,T]\rightarrow \mathbb{R}^g$, velocities $v:[0,T]\rightarrow \mathbb{R}^g$ and a Lagrange multiplier $\lambda:[0,T]\rightarrow \mathbb{R}$, see Figure \ref{MultiBodyPic}. In the chain of mass points the $i$th mass of weight $m_i$ is connected to the $(i+1)st$ mass by a spring and a damper with constants $k_i$ and $d_i$ and also to the ground by a spring and a damper with the constants $\kappa_i$ and $\delta_i$, respectively, where $m_i$, $k_i$, $d_i$, $\kappa_i$, $\delta_i>0$. Furthermore, the first and the last mass points are connected by a rigid bar. The vibrations are driven by an external force $u:[0,T]\rightarrow\mathbb{R}$ (control input) acting on the first mass point. The resulting DAE system of size $n=2g+1$ is not port-Hamiltonian and has differentiation-index three. In first order formulation it is given by
\begin{align*}
\begin{bmatrix} I & 0 & 0\\ 0 & M & 0\\ 0&0&0\end{bmatrix}\begin{bmatrix}\dot{p}\\ \dot{v}\\ \dot{\lambda}\end{bmatrix}&=\begin{bmatrix}0&I&0\\K&D&-G^T\\
G&0&0\end{bmatrix}\begin{bmatrix}
p\\v\\\lambda\end{bmatrix}+\begin{bmatrix}0\\F\\0\end{bmatrix} u
\end{align*}
with mass matrix $M=\mathrm{diag}(m_1,\dots,m_g)$, tridiagonal stiffness and damping matrices $K$, $D\in \mathbb{R}^{g\times g}$, constraint matrix $G=[1\,\,0 \dots 0\,\,-1]\in \mathbb{R}^{1\times g}$ and input matrix $F=e_1\in \mathbb{R}^{g\times 1}$. Assuming $K$ and $D$ to be symmetric negative semi-definite, the multibody problem can be formulated as a pHDAE of differentiation-index two by replacing the algebraic constraint $Gp=0$ by its first derivative $Gv=0$ and adding an appropriate output equation,
\begin{equation}
\begin{aligned}\label{MultiBodyPH}
\underbrace{\begin{bmatrix} I & 0 & 0\\ 0 & M & 0\\ 0&0&0\end{bmatrix}}_E\begin{bmatrix}\dot{p}\\ \dot{v}\\ \dot{\lambda}\end{bmatrix}&=\left(\underbrace{\begin{bmatrix}0&I&0\\-I&0&-G^T\\
0&G&0\end{bmatrix}}_J-\underbrace{\begin{bmatrix}0&0&0\\0&-D&0\\0&0&0\end{bmatrix}}_R\right)
\underbrace{\begin{bmatrix}
-K&0&0\\0&I&0\\0&0&I
\end{bmatrix}}_Q \begin{bmatrix}
p\\v\\\lambda\end{bmatrix}+\underbrace{\begin{bmatrix}0\\F\\0\end{bmatrix}}_B u\\
y&=B^TQx.
\end{aligned}
\end{equation}
Obviously, $J=-J^{T}$, $R=R^T\geq 0$,  $E^TQ=Q^TE\geq 0$ and $Q^TRQ\geq 0$ hold. But note that $(J-R)Q$ is singular such that for $s=0$ the matrix  $(s E-(J-R)Q)$ is not invertible.
\begin{figure}[t]
\includegraphics[scale=0.4]{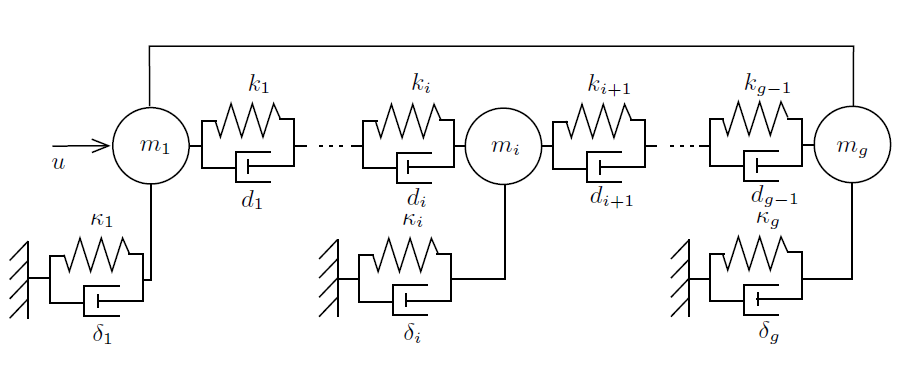}
\caption{Damped mass-spring system with holonomic constraint \cite{MehS05}}\label{MultiBodyPic}
\end{figure}

The structure of the equations simplify a further index reduction for \eqref{MultiBodyPH} analogous to that for the flow problems. The singular value decomposition of $G^T$, i.e., $GV=[Z\,\, 0]$ with orthogonal matrix $V=[V_1 \,\, V_2]$, $V_1\in\mathbb{R}^{g\times 1}$, $V_2\in\mathbb{R}^{g \times (g-1)}$ as well as $Z\in \mathbb{R}^{1\times 1}$ invertible, yields
 \begin{align*}\label{MultiBodyPHReg}
\begin{bmatrix} I & 0 &0& 0\\ 0 & M_{11} & M_{12} & 0\\ 0& M_{12}^T&M_{22}&0\\0& 0&0&0\end{bmatrix}\begin{bmatrix}\dot{p}\\ \dot{v}_1\\ \dot{v}_2\\ \dot{\lambda}\end{bmatrix}&=\begin{bmatrix}0&V_1&V_2&0\\-V_1^T&D_{11}&D_{12}&-Z^T\\-V_2^T&D_{12}^T&D_{22}&0\\0&Z&0&0\end{bmatrix}\begin{bmatrix}
-K&0&0&0\\0&I&0&0\\0&0&I&0\\0&0&0&I
\end{bmatrix}\begin{bmatrix}
p\\v_1\\v_2\\\lambda\end{bmatrix}+\begin{bmatrix}0\\B_1\\B_2\\0\end{bmatrix} u.
\end{align*}
The last equation $Zv_1=0$ implies that $v_1=0$. Differentiating this equation and inserting it into the second equation yields the hidden constraint for the Lagrange multiplier
\begin{align*}
{Z}^{T}\lambda&=-M_{12}\dot{v}_2+V_1^TKp+D_{12}v_2+B_1u,
\end{align*}
which also imposes a consistency condition for the initial value.
The underlying ODE of size $n_1=2g-1$ together with the output equation are given by
\begin{equation}
\begin{aligned}\label{MBSuODE}
\begin{bmatrix}I&0\\0&M_{22}\end{bmatrix}\begin{bmatrix}\dot{p}\\\dot{v}_2
\end{bmatrix}&=\left(\begin{bmatrix}0&V_2\\-V_2^T&0\end{bmatrix}-\begin{bmatrix}0&0\\0&-D_{22}\end{bmatrix}\right)\begin{bmatrix}
-K&0\\0&I\end{bmatrix}\begin{bmatrix}p\\v_2\end{bmatrix}+\begin{bmatrix}
0\\B_2\end{bmatrix}u,\\
y&=\begin{bmatrix}
0 & B_2^T\end{bmatrix} \begin{bmatrix}
-K&0\\0&I\end{bmatrix} \begin{bmatrix}p\\v_2\end{bmatrix}.
\end{aligned}
\end{equation}

We present numerical results for an example setup, where the parameters of the mass-spring system are set to be $m_i=100$ for $i=1,...,g$ and $k_i=\kappa_j=2$, $d_i=\delta_j=5$ for $i=1,...,g-1$, $j=2,...,g-1$ as well as $\kappa_1=\kappa_g=4$, and $\delta_1=\delta_g=10$. Choosing $g=6000$ yields a state of size $n=12001$ for the original DAE and of $n_1=11999$ for the underlying ODE to be reduced. Since the dimension is odd, the skew-symmetric interconnection matrix is singular and also the relevant submatrix for FCRM is singular as well. But MM at almost all $s_0\in\mathbb{C}\backslash\{0\}$ and ECRM can be applied. Figure \ref{MBSNorms} shows the respective relative errors of the transfer functions in spectral norm for the reduced size $r=10$ and in the $H_\infty$-norm for $r\in[2,20]$. As expected, MM at $s_0=\infty$ and $s_0=10^{-10}$ yields outstanding approximations (errors of order $\mathcal{O}(10^{-15})$) for high and low frequencies, respectively. ECRM, in contrast, provides a uniformly good approximation quality of order $\mathcal{O}(10^{-5})$, independently of the chosen frequency. In the $H_\infty$- and the $H_2$-norms it even outperforms the moment matching versions by one up to two orders.
\begin{figure}[t]
\includegraphics[width=0.49\textwidth]{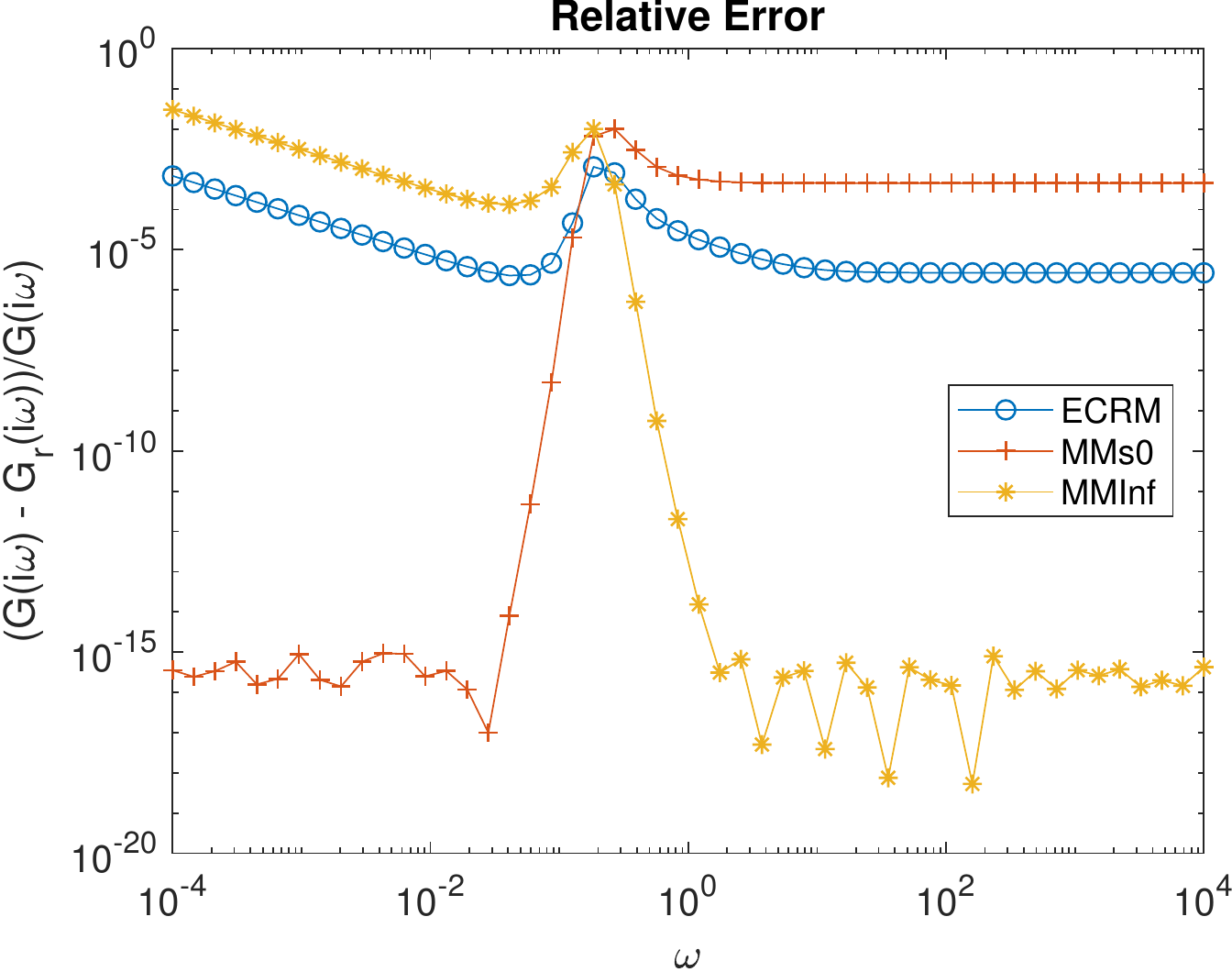}\hfill
\includegraphics[width=0.49\textwidth]{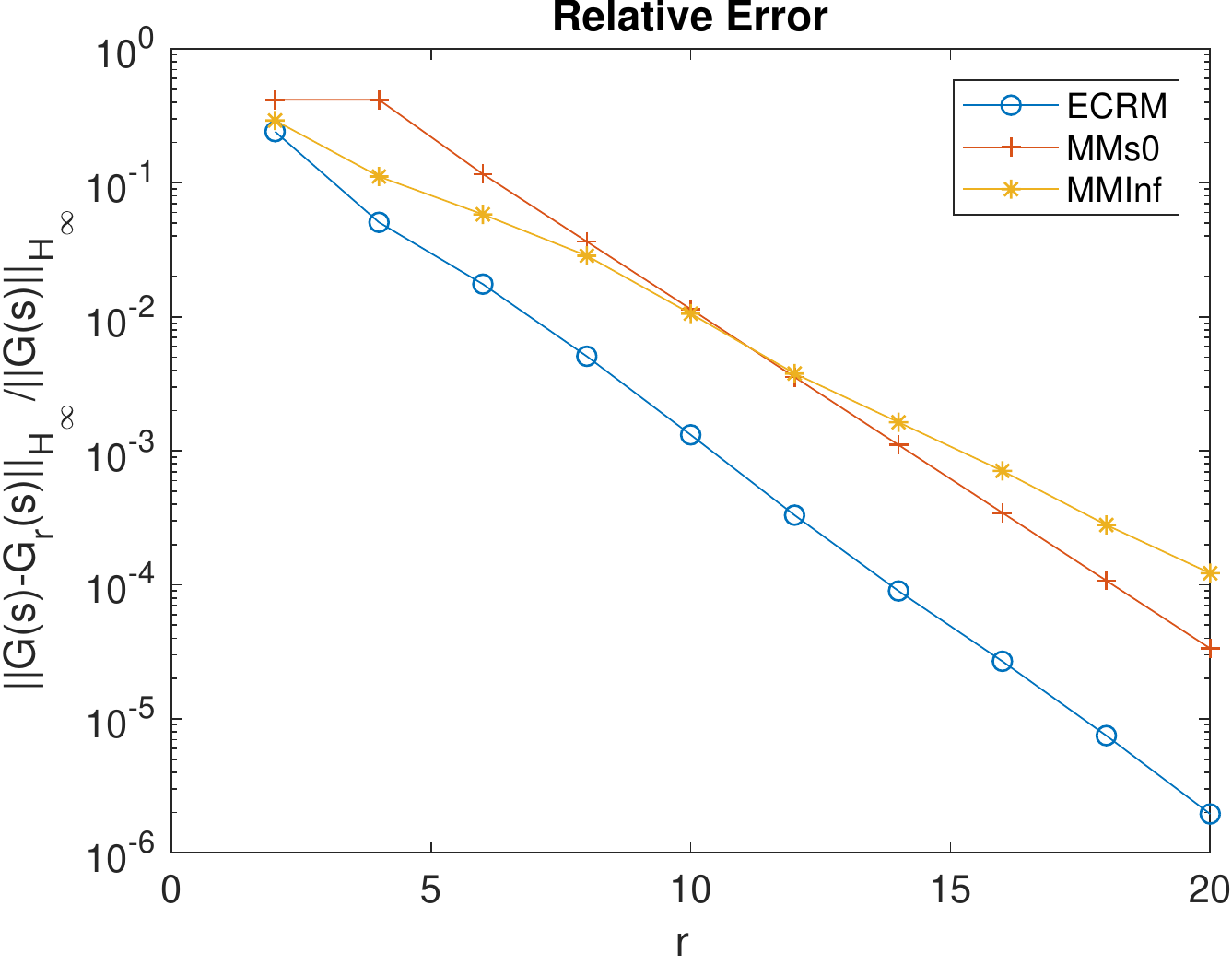}\\
\includegraphics[width=0.49\textwidth]{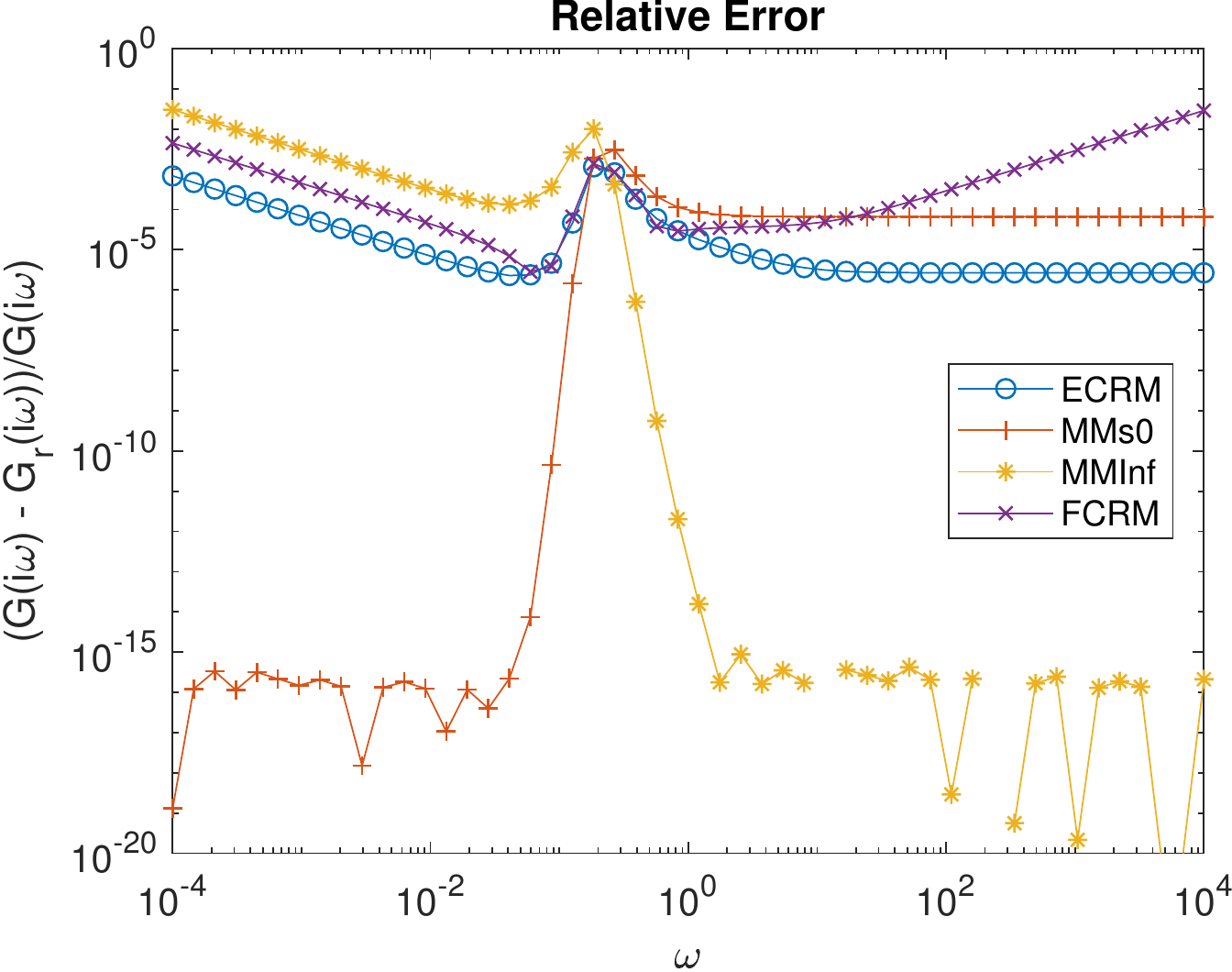}\hfill
\includegraphics[width=0.49\textwidth]{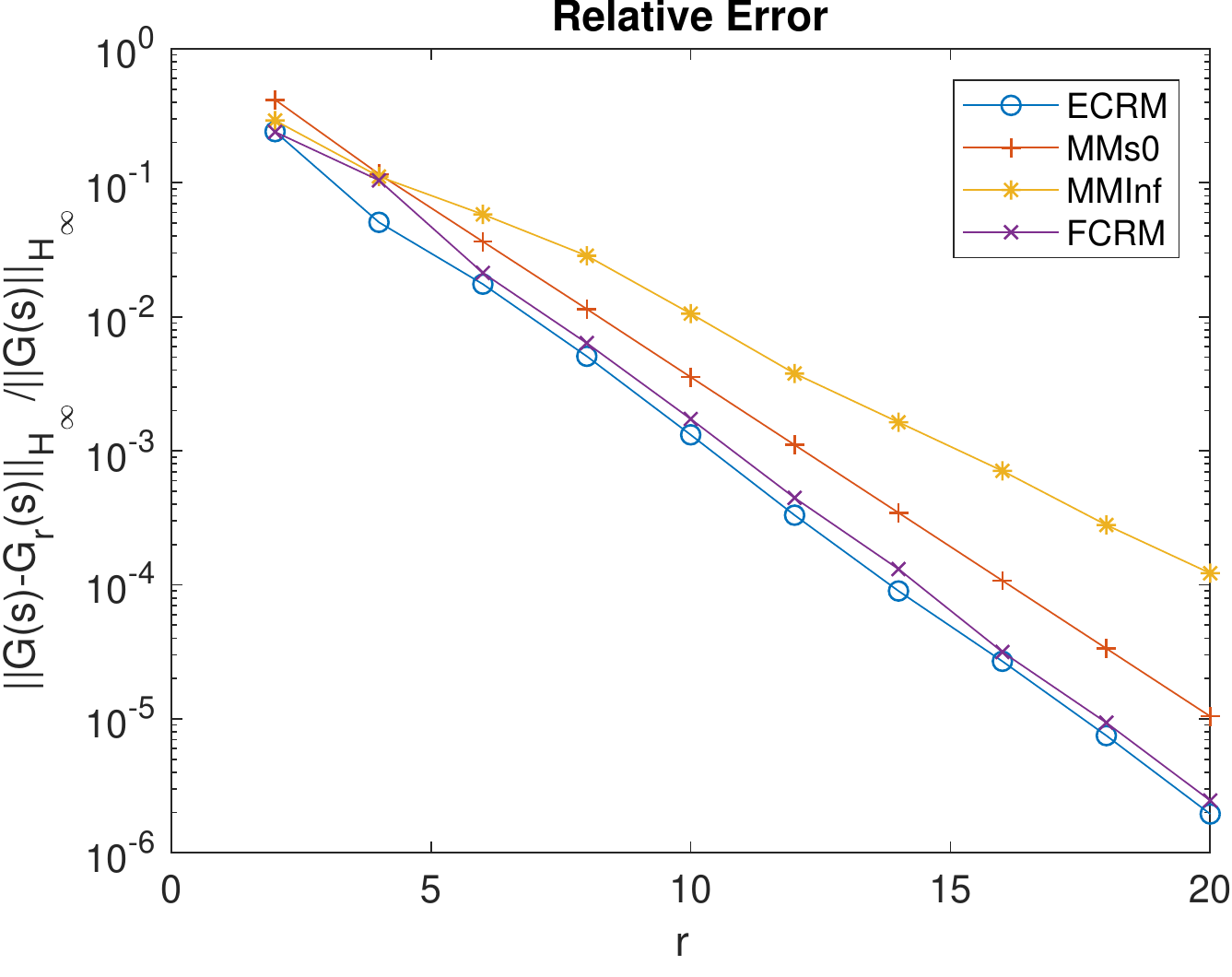}
\caption{Mass-spring system: index-reduced formulation \eqref{MultiBodyPH}, $n=2g+1$ \textit{(top)} and index-reduced formulation via minimal extension \eqref{MultiBodyME}, $n=2g+2$ \textit{(bottom)}.
Relative errors of reduced transfer functions plotted over frequency \textit{(left)} and in $H_\infty$-norm over reduced state size $r$ \textit{(right)}. (MM is performed at $s_0=\infty$ as well as at $s_0=10^{-10}$ for \eqref{MultiBodyPH} and at $s_0=0$ for \eqref{MultiBodyME}.)  }
\label{MBSNorms}
\end{figure}
\begin{remark}{\rm 
Alternatively to \eqref{MultiBodyPH}, the damped mass-spring system can be also formulated by keeping the original constraint $Gp=0$ and adding the additional constraint $Gv=0$ to match the symmetry structure. This  is called index reduction by minimal extension, see \cite{KunM06}, and reduces the differentiation-index to two in  the system
\begin{align}\label{MultiBodyME}
\begin{bmatrix} I & 0 & 0 & 0\\ 0 & M & 0 & 0 \\ 0&0&0&0\\0&0&0&0\end{bmatrix}\begin{bmatrix}\dot{p}\\\dot{v}\\\dot{\lambda_1}\\\dot{{\lambda_2}}\end{bmatrix}&
=\begin{bmatrix}0&I&0&-G^T\\K&D&-G^T&0\\0&G&0&0\\  G&0&0&0
\end{bmatrix}
\begin{bmatrix}p\\v\\\lambda_1\\{\lambda_2}
\end{bmatrix}+\begin{bmatrix}0\\B\\0\\0\end{bmatrix}u.
\end{align}
The resulting system \eqref{MultiBodyME} is of size $n=2g+2$, has a regular system matrix and a regular matrix pencil, but is not in port-Hamiltonian form. However, using again the singular value decomposition of $G^T$, solving the last equation and inserting the derivative yields a pHODE for $[p_2^T\,\, v_2^T]:[0,T]\rightarrow \mathbb{R}^{n_1}$, $n_1=2(g-1)$  of the form
\begin{equation}
\begin{aligned}\label{MbsuODE2}
\begin{bmatrix}I&0\\0&M_{22}\end{bmatrix}\begin{bmatrix}\dot{p}_2\\\dot{v}_2\end{bmatrix}&=
\left(\begin{bmatrix}0&I\\-I&0\end{bmatrix}-\begin{bmatrix}0&0\\0&-D_{22}\end{bmatrix}\right)\begin{bmatrix}-K_{22}&0\\0&I\end{bmatrix}\begin{bmatrix}p_2\\v_2\end{bmatrix}+
\begin{bmatrix}0\\B_2\end{bmatrix}u,\\
y&=\begin{bmatrix}
0 & B_2^T\end{bmatrix} \begin{bmatrix}
-K_{22}&0\\0&I\end{bmatrix} \begin{bmatrix}p_2\\v_2\end{bmatrix}
\end{aligned}
\end{equation}
whose interconnection matrix is invertible.
The other variables satisfy $v_1=p_1=\lambda_2=0$ and $Z^T\lambda_1=-M_{12}\dot{v}_2+K_{12}p_2+D_{12}v_2+B_{1}u$. To this formulation, also FCRM and MM at $s_0=0$ are applicable. While ECRM and MM at $s_0=\infty$ yield analogous results independent of the problem formulation, MM at $s_0=0$ for \eqref{MbsuODE2} provides a slightly better $H_\infty$-approximation than MM at $s_0=10^{-10}$ for \eqref{MBSuODE}. FCRM shows in general a similar approximation behavior as ECRM but suffers from an error drift off for high frequencies caused by its additional feed-through terms (cf.\ Figure~\ref{MBSNorms}).
}
\end{remark}

\section{Conclusion}\label{sec:5}
\noindent
The power conservation methods (ECRM and FCRM) as well as moment matching via Galerkin projections are established structure-preserving model reduction techniques for standard port-Hamiltonian systems of ordinary differential equations. In this paper we have adapted them to handle also port-Hamiltonian differential-algebraic systems of differentiation-index one or two. Making use of an appropriate decoupling of differential and algebraic variables, the dynamic state is reduced, while the properties and all explicit and hidden constraints of the pHDAE are preserved. The performance of the techniques has been illustrated for benchmark problems arising from  spatially discretized flow problems and multibody systems. ECRM shows similarities to Balanced Truncation, if a Lyapunov balancing is performed. Therefore, as expected, ECRM outperforms moment matching when studying the reduction errors in $H_\infty$- and/or $H_2$-norms, whereas moment matching yields better local approximations in the spectral norm.
The performance of FCRM is comparable to ECRM, but it may suffer from an error increase for high frequencies caused by the feed-through terms generated in the reduced model. Moreover, its applicability is limited.

\subsection*{Acknowledgements}
\noindent The authors acknowledge the support by the German BMBF, Project EiFer, and the German BMWi, Project MathEnergy.



\begin{thebibliography}{10}

\bibitem{AliBMSV17}
{\sc N.~Aliyev, P.~Benner, E.~Mengi, P.~Schwerdtner, and M.~Voigt}, {\em
  Large-scale computation of {$\mathcal{L}_\infty$}-norms by a greedy subspace
  method}, SIAM Journal on Matrix Analysis and Applications, 38 (2017),
  pp.~1496--1516.

\bibitem{Ant05}
{\sc A.~C. Antoulas}, {\em Approximation of Large-Scale Dynamical Systems},
  {SIAM} Publications, Philadelphia, 2005.

\bibitem{BeaGM17}
{\sc C.~Beattie, S.~Gugercin, and V.~Mehrmann}, {\em Model reduction for
  large-scale dynamical systems with inhomogeneous initial conditions}, Systems
  and Control Letters, 99 (2017), pp.~99--106.

\bibitem{BeaMXZ18}
{\sc C.~Beattie, V.~Mehrmann, H.~Xu, and H.~Zwart}, {\em Linear
  port-{H}amiltonian descriptor systems}, Mathematics of Control, Signals, and
  Systems, 30 (2018), pp.~1--27.

\bibitem{BenMS05}
{\sc P.~Benner, V.~Mehrmann, and D.~C. Sorensen}, eds., {\em {D}imension
  {R}eduction of {L}arge-{S}cale {S}ystems}, vol.~45 of Lecture Notes in
  Computational Science and Engineering, Springer, Berlin, 2005.

\bibitem{BenS06}
{\sc P.~Benner and V.~I. Sokolov}, {\em Partial realization of descriptor
  systems}, Systems and Control Letters, 55 (2006), pp.~929--938.

\bibitem{BenS17}
{\sc P.~Benner and T.~Stykel}, {\em Model order reduction for
  differential-algebraic equations: A survey}, in Surveys in
  Differential-Algebraic Equations. {IV}, Springer, Heidelberg, 2017,
  pp.~107--160.

\bibitem{BorG15}
{\sc J.~Borggaard and S.~Gugercin}, {\em Model reduction for {DAE}s with an
  application to flow control}, in Active Flow and Combustion Control 2014,
  Springer, Cham, 2015, pp.~381--396.

\bibitem{BreCP96}
{\sc K.~E. Brenan, S.~L. Campbell, and L.~R. Petzold}, {\em Numerical Solution
  of Initial-Value Problems in Differential Algebraic Equations}, {SIAM}
  Publications, Philadelphia, 2nd~ed., 1996.

\bibitem{ByeGM97}
{\sc R.~{Byers}, T.~{Geerts}, and V.~{Mehrmann}}, {\em Descriptor systems
  without controllability at infinity}, SIAM Journal on Control and
  Optimization, 35 (1997), pp.~462--479.

\bibitem{CamKM12}
{\sc S.~L. {Campbell}, P.~{Kunkel}, and V.~{Mehrmann}}, {\em Regularization of
  linear and nonlinear descriptor systems}, in Control and Optimization with
  Differential-Algebraic Constraints, no.~23 in Advances in Design and Control,
  Philadelphia, PA, 2012, {SIAM} Publications, Philadelphia, pp.~17--36.

\bibitem{ChaBG16}
{\sc S.~Chaturantabut, C.~Beattie, and S.~Gugercin}, {\em Structure-preserving
  model reduction for nonlinear port-{H}amiltonian systems}, SIAM Journal on
  Scientific Computing, 38 (2016), pp.~B837--B865.

\bibitem{EggKLMM18}
{\sc H.~Egger, T.~Kugler, B.~Liljegren-Sailer, N.~Marheineke, and V.~Mehrmann},
  {\em On structure-preserving model reduction for damped wave propagation in
  transport networks}, SIAM Journal on Scientific Computing, 40 (2018),
  pp.~A331--A365.

\bibitem{EmmM13}
{\sc E.~{Emmrich} and V.~{Mehrmann}}, {\em Operator differential-algebraic
  equations arising in fluid dynamics}, Computer Methods in Applied
  Mathematics, 13 (2013), pp.~443--470.

\bibitem{Fre05}
{\sc R.~W. Freund}, {\em Pad\'e-type model reduction of second-order and
  higher-order linear dynamical systems}, in Dimension Reduction of Large-Scale
  Systems, Springer, Berlin, 2005, pp.~191--223.

\bibitem{GolV96}
{\sc G.~H. Golub and C.~F. {Van Loan}}, {\em Matrix Computations}, The Johns
  Hopkins University Press, Baltimore, MD, 3rd~ed., 1996.

\bibitem{GraMQSW16}
{\sc N.~Gr{\"a}bner, V.~Mehrmann, S.~Quraishi, C.~Schr\"oder, and U.~{von
  W}agner}, {\em Numerical methods for parametric model reduction in the
  simulation of disc brake squeal}, Zeitschrift f{\"u}r Angewandte Mathematik
  und Mechanik, 96 (2016), pp.~1388--1405.

\bibitem{GugPBS12}
{\sc S.~Gugercin, R.~V. Polyuga, C.~Beattie, and A.~{van der S}chaft}, {\em
  Structure-preserving tangential interpolation for model reduction of
  port-{H}amiltonian systems}, Automatica, 48 (2012), pp.~1963--1974.

\bibitem{GugPBS09}
{\sc S.~Gugercin, R.~V. Polyuga, C.~A. Beattie, and A.~van~der Schaft}, {\em
  Interpolation-based {H2} model reduction for port-{H}amiltonian systems}, in
  Proceedings 48th IEEE Conference on Decision and Control, {IEEE}, 2009,
  pp.~5362--5369.

\bibitem{GugSW13}
{\sc S.~Gugercin, T.~Stykel, and S.~Wyatt}, {\em Model reduction of descriptor
  systems by interpolatory projection methods}, SIAM Journal on Scientific
  Computing, 35 (2013), pp.~B1010--B1033.

\bibitem{HeiSS08}
{\sc M.~Heinkenschloss, D.~C. Sorensen, and K.~Sun}, {\em Balanced truncation
  model reduction for a class of descriptor systems with application to the
  {O}seen equations}, SIAM Journal on Scientific Computing, 30 (2008),
  pp.~1038--1063.

\bibitem{KotL18}
{\sc P.~Kotyczka and L.~Lef{\`e}vre}, {\em Discrete-time port-{H}amiltonian
  systems: A definition based on symplectic integration}, IFAC-Papers OnLine,
  51 (2018), pp.~125--130.

\bibitem{KunM06}
{\sc P.~Kunkel and V.~Mehrmann}, {\em Differential-Algebraic Equations},
  European Mathematical Society (EMS), Z{\"u}rich, 2006.

\bibitem{MehMW18}
{\sc C.~Mehl, V.~Mehrmann, and M.~Wojtylak}, {\em Linear algebra properties of
  dissipative port-{H}amiltonian descriptor systems}, SIAM Journal on Matrix
  Analysis and Applications, 39 (2018), pp.~1489--1519.

\bibitem{MehS05}
{\sc V.~Mehrmann and T.~Stykel}, {\em Balanced truncation model reduction for
  large-scale systems in descriptor form}, in Dimension Reduction of
  Large-Scale Systems, Springer, Berlin, 2005, pp.~83--115.

\bibitem{PolS10}
{\sc R.~V. Polyuga and A.~van~der Schaft}, {\em Structure-preserving model
  reduction of port-{H}amiltonian systems by moment matching at infinity},
  Automatica, 46 (2010), pp.~665--672.

\bibitem{PolS11}
\leavevmode\vrule height 2pt depth -1.6pt width 23pt, {\em Structure-preserving
  moment matching for port-{H}amiltonian systems: {A}rnoldi and {L}anczos},
  IEEE Transactions on Automatic Control, 56 (2011), pp.~1458--1462.

\bibitem{PolS12}
\leavevmode\vrule height 2pt depth -1.6pt width 23pt, {\em Effort- and
  flow-constraint reduction methods for structure-preserving model reduction of
  port-{H}amiltonian systems}, Systems and Control Letters, 61 (2012),
  pp.~412--421.

\bibitem{Ria08}
{\sc R.~Riaza}, {\em Differential-Algebraic Systems. Analytical Aspects and
  Circuit {A}pplications}, World Scientific Publishing Co. Pte. Ltd.,
  Hackensack, NJ., 2008.

\bibitem{Saa03}
{\sc Y.~Saad}, {\em Iterative Methods for Sparse Linear Systems}, {SIAM}
  Publications, Philadelphia, 2nd~ed., 2003.

\bibitem{SakKB16}
{\sc J.~Saak, M.~Köhler, and P.~Benner}, {\em {M-M.M.E.S.S. 1.0.1} -- {T}he
  matrix equations sparse solvers library}, 2016.
\newblock see also: www.mpi-magdeburg.mpg.de/projects/mess.

\bibitem{Sch17_ppt}
{\sc L.~Scholz}, {\em Condensed forms for linear port-{H}amiltonian descriptor
  systems}, Preprint 09--2017, Institut f\"ur Mathematik, TU Berlin, 2017.

\bibitem{Sty06}
{\sc T.~Stykel}, {\em Balanced truncation model reduction for semidiscretized
  {S}tokes equation}, Linear Algebra and its Applications, 415 (2006),
  pp.~262--289.

\bibitem{Sty06b}
\leavevmode\vrule height 2pt depth -1.6pt width 23pt, {\em On some norms for
  descriptor systems}, IEEE Transactions on Automatic Control, 51 (2006),
  pp.~842--847.

\bibitem{Sch13}
{\sc A.~{van der S}chaft}, {\em Port-{H}amiltonian differential-algebraic
  systems}, in Surveys in Differential-Algebraic Equations. {I}, Springer,
  Heidelberg, 2013, pp.~173--226.

\bibitem{SchJ14}
{\sc A.~{van der S}chaft and D.~Jeltsema}, {\em Port-{H}amiltonian systems
  theory: {A}n introductory overview}, Foundations and Trends in Systems and
  Control, 1 (2014), pp.~173--378.

\bibitem{SchM18}
{\sc A.~{van der S}chaft and B.~Maschke}, {\em Generalized port-{H}amiltonian
  {DAE} systems}, Systems and Control Letters, 121 (2018), pp.~31--37.

\bibitem{Var91}
{\sc A.~Varga}, {\em Balancing free square-root algorithm for computing
  singular perturbation approximations}, in Proceedings 30th IEEE Conference on
  Decision and Control, {IEEE}, 1991, pp.~1062--1065 vol.2.

\bibitem{WolLEK10}
{\sc T.~Wolf, B.~Lohmann, R.~Eid, and P.~Kotyczka}, {\em Passivity and
  structure preserving order reduction of linear port-{H}amiltonian systems
  using {K}rylov subspaces}, European Journal of Control, 16 (2010),
  pp.~401--406.

\end{thebibliography}
\end{document}